\documentclass[11pt,a4paper,reqno]{amsart} 
\usepackage{a4wide,ifthen,amssymb,color}
\usepackage{mathtools}
\usepackage{tikz}
\mathtoolsset{showonlyrefs}
\usepackage[all]{xypic}
\usepackage{mathrsfs}
\usepackage{amsmath}

\usepackage[backref=page]{hyperref}

\hypersetup{
 colorlinks,
 citecolor=red,
 linkcolor=blue,
 urlcolor=Blue}

\newcommand{\qee}{\parbox{5pt}{\hfill}\hfill $\triangle$}
\newtheorem{theorem}{Theorem}[section] 
\newtheorem{corollary}[theorem]{Corollary}
\newtheorem{lemma}[theorem]{Lemma} 
\newtheorem{proposition}[theorem]{Proposition}
\newtheorem{definition}[theorem]{Definition}
\theoremstyle{definition}

\theoremstyle{remark}
\newtheorem{remarkee}[theorem]{Remark}
\newtheorem{examplee}[theorem]{Example}
\newenvironment{example}{\begin{examplee}}{\qee\end{examplee}}
\newenvironment{remark}{\begin{remarkee}}{\qee\end{remarkee}}
\numberwithin{equation}{section}

\newcommand\rk{\operatorname{rk}}

\newcommand\Hom{\operatorname{Hom}}
\newcommand\iso{\kern.35em{\raise3pt\hbox{$\sim$}\kern-1.1em\to}\kern.3em}

\makeatletter
\renewcommand*\subjclass[2][2020]{%
  \def\@subjclass{#2}%
  \@ifundefined{subjclassname@#1}{%
    \ClassWarning{\@classname}{Unknown edition (#1) of Mathematics
      Subject Classification; using '2020'.}%
  }{%
    \@xp\let\@xp\subjclassname\csname subjclassname@#1\endcsname
  }%
}
\@namedef{subjclassname@1991}{%
  \textup{1991} Mathematics Subject Classification}
\@namedef{subjclassname@2000}{%
  \textup{2000} Mathematics Subject Classification}
\@namedef{subjclassname@2010}{%
  \textup{2010} Mathematics Subject Classification}
\@namedef{subjclassname@2020}{%
  \textup{2020} Mathematics Subject Classification}
\@xp\let\@xp\subjclassname\csname subjclassname@2020\endcsname

\makeatother

\begin{document}
\title[Exceptional set of crepant resolutions]{On the exceptional set of crepant  \\[3pt] resolutions of   abelian   singularities}
\author[U. Bruzzo]{Ugo  Bruzzo$^\P$}
\author[F. Arceu Ferreira]{F\'abio  Arceu Ferreira$^\ddag$}

\email{bruzzo@sissa.it}
\email{fabio.arceu.ferreira@gmail.com}

\address{\P\ SISSA (Scuola Internazionale Superiore di Studi Avanzati), Via Bonomea 265, 34136, Trie\-ste, Italia;
INFN (Istituto Nazionale di Fisica Nucleare), Sezione di Trieste; 
IGAP (Institute for Geometry and Physics), Trieste; 
}

\address{\ddag\ PPGMAT, Departamento de Matemática, Universidade Federal da Paraíba, Campus I, Castelo Branco, João Pessoa,  PB 58051-900, Brazil}

\date{\today}

\thanks{\vskip-12pt\P\ Research partly supported by GNSAGA-INdAM and by the PRIN project 2022BTA242  ``Geometry of algebraic structures: moduli, invariants, deformations.''}
\thanks{\ddag\  PhD student at UFPB from 2021 to 2025 supported by FAPESQ-PB (Edital n$^\circ$ 16/2022) and visiting student at SISSA in 2023/2024 supported by CAPES PDSE-202324285975.}

\subjclass{Primary: 14E15; Secondary:
14B05, 14J27, 14M25} 
\keywords{Crepant resolutions, exceptional set, toric geometry, normal bundle}

\begin{abstract} Let \( G \) be a finite abelian subgroup of \( \operatorname{SL}(n, \mathbb{C}) \), and suppose there exists a toric crepant resolution 
$
\phi: X \longrightarrow \mathbb{C}^n / G
$
of the quotient variety \( \mathbb{C}^n / G \). Let 
$
\operatorname{Exc}(\phi) = E_1 \cup \dots \cup E_s
$
 be the decomposition of the exceptional set of $\phi$ into irreducible components. We prove that  for every  $i$  there exists an open subset \( U_i \) of \( X \) such that \( E_i \subset U_i \), and \( U_i \) is isomorphic to the total space of the canonical bundle \( \omega_{E_i} \) of \( E_i \). Furthermore, 
$
X = U_1 \cup \dots \cup U_s.
$
This  contributes  to the collection of results aimed at solving a classical problem, i.e., to determine which submanifolds of a complex manifold have a neighborhood isomorphic to a neighborhood of the zero section of their normal bundle. 
 \end{abstract}
\maketitle 

\tableofcontents

\section{Introduction}
Quotient singularities  $ \mathbb{C}^n / G $, where $ G $ is a finite subgroup of $ \operatorname{SL}(n, \mathbb{C}) $, are an interesting object of study in algebraic geometry due to their rich structure and the various applications in both mathematics and physics (for the latter, see among others \cite{Bruzzo19,Bruzzo19bis,Bianchi21}). As $ G $ is a finite subgroup of $ \operatorname{SL}(n,\mathbb{C}) $, the quotient variety $ \mathbb{C}^n / G $ is a Gorenstein singularity \cite{Yau-Yu,Watanabe74}.
%
For $n=2$ the varieties $\mathbb C^2/G$ 
are known as
Kleinian or Du Val singularities, and are classified by the \( A_n, D_n, \) and \( E_n \) series; they are  deeply connected with the representation theory of simple Lie algebras \cite{Reid02}. 

Given a resolution of singularities $f:X\to\mathbb{C}^{n}/G$, one has the ramification formula 
\begin{eqnarray*}
K_{X}=f^{*} K_{\mathbb{C}^{n}/G}+\sum_{i=1}^{s}a_{i}E_{i},    
\end{eqnarray*}
where $K_{X}$ is a canonical divisor of $X$, $K_{\mathbb{C}^{n}/G}$ is a canonical divisor of $\mathbb{C}^{n}/G$, the divisors $E_{1},\dots, E_{s}$ are  the irreducible components of the exceptional locus $\operatorname{Exc}(f)$ of $f$, and $a_{1},\dots,a_{s}$ are integer numbers. $f$ is said to be a crepant resolution of singularities if $a_{i}=0$ for every $i=1,\dots, s.$  The existence of crepant resolutions is a fundamental problem for quotient singularities. 

For $n=2$  it is known that $\mathbb{C}^{2}/G$ admits a unique  crepant resolution of singularities whose  prime exceptional divisors are all isomorphic to $\mathbb{P}^{1}$ \cite{Durfee79}. 

For \( n = 3 \), the problem of the existence of a crepant resolution was at first  solved   case by case \cite{Ito95,Ito95b,Markushevich97, Roan94, Roan96}, since the conjugacy classes of finite subgroups of $\operatorname{SL}(3,\mathbb{C})$ were
listed (\cite{Yau-Yu}, see also \cite{Carrasco14} for a very clear exposition). Although a  crepant resolution is not necessarily unique, there is a canonical resolution, as shown in \cite{Bridgeland01}, which is the \( G \)-Hilbert scheme $\mathrm{Hilb}^G(\mathbb{C}^3)$, the space  parameterizing all \( G \)-invariant zero-dimensional subschemes of \( \mathbb{A}^3 \) of length 
$
 \#G
$
(the order of \( G \)); the space of global sections of  their structure sheaf is  isomorphic to the regular representation of the group \( G \) as \( G \)-modules. This canonical crepant resolution had already been found when $G$ is   abelian      using   toric geometry \cite{Nakamura01}.
Indeed when  \( G   \) is abelian, \( \mathbb{C}^n / G \) is a toric variety, and toric geometry provides powerful tools to describe both the singularity and its resolutions. 

More generally, the McKay correspondence establishes a striking relationship between the geometry of a crepant resolution and the representation theory of \( G \). This correspondence has been extensively studied in dimensions \( n = 2 \) \cite{Verdier84} and \( n = 3 \) \cite{ItoReid94} and continues to motivate research in higher dimensions \cite{Bridgeland01}.
 
In this paper we  prove the following result. Let $\{g_{1},\dots,g_{s}\}$ be the set of junior classes of $G$. Suppose that there is a Hilbert basis resolution $\phi:X_{\Xi,N}\to\mathbb{C}^{n}/G=U_{\sigma,N}$ of the quotient variety $\mathbb{C}^{n}/G$. The fractional expressions $\hat{g_{1}},\dots,\hat{g}_{s}$ are elements of $\bold{Hlb}_{N}(\sigma)$. Thus  $E_{g_{i}} = V(\operatorname{Cone}(\hat{g_{i}}))$ is a exceptional prime divisor of the resolution  for every $i=1,\dots, s$. We prove that there is an open toric set $U_{i}$ of $X_{\Xi}$ that contains $E_{g_{i}}$ together with a torus invariant divisor $D_{i}$ of $E_{g_{i}}$ such that an isomorphism $\varphi: \operatorname{tot}(\mathcal{O}_{E_{g_{i}}}(D_{i}))\to U_{i}$  exists,  and   is the identity on  the zero section. In particular, when $\mathbb{C}^{n}/G$ admits a crepant resolution we get:  

\noindent\textbf{Theorem \ref{t3.4}.} {\em Let $ G $ be an abelian finite subgroup of $ \operatorname{SL}(n,\mathbb{C}) $ and suppose that there exists a crepant resolution $ \mathbb{\phi}:X_{\Xi}\rightarrow \mathbb{C}^{n}/G=U_{\sigma,N} $ of the quotient variety. If $ g $ is a junior element of $ G $, then $ E_{g} $ is normally embedded in $ X_{\Xi} $. In particular, the total space of the canonical bundle of $ E_{g} $, $ \operatorname{tot}(\omega_{E_{g}}) $, is isomorphic to the toric variety $ X_{\Xi_{g}} $, and
\begin{equation*}
X_{\Xi}=\bigcup_{\hat{g}\in\hat{G}\cap\bigtriangleup_{1}}X_{\Xi_{g}}
\end{equation*}
where $ \Xi_{g} $ is the fan consisting of all the faces of the cones that appear in the set
\begin{equation*}
\Xi_{g}(n):=\{\eta\in\Xi(n)|\rho_{g}\preceq\eta\}.
\end{equation*}
In particular, $ X_{\Xi_{g}} $ is open in $ X_{\Xi} $.}
 
  ``Normally embedded" 
means that $X_{\Xi_{g}}$ is a tubular neighborhood of $E_{g}$ (i.e. it is isomorphic to the total space of the normal bundle $\mathcal{N}_{E_{g}}/X_{\Xi}$). So, this result contributes to the collection of results aimed at solving a classical problem: Determine which subvarieties of algebraic varieties have a neighborhood, isomorphic to the neighborhood of the zero section of their normal bundle. This problem is trivial in the category of real manifolds (see \cite{Hirsch76}, Chapter $4$), but is highly nontrivial in the holomorphic and algebraic settings \cite{MorrowRossi78}. The problem has been
widely studied, starting with the works of Grauert \cite{Grauert62} and Van de Ven \cite{VandeVen59};
see  \cite{AbateBracciTovena09}  for related bibliography. Our Theorem \ref{t3.4} provides a partial solution to a global version of this problem, providing a class of examples in which the subvariety has a neighborhood isomorphic to the whole of its normal bundle, and not just to a neighborhood of its zero section. 

We briefly describe the content of this paper. 
 Section $2$ is a review of known results. 
 We discuss the properties of the quotient variety $\mathbb{C}^{n}/G$ in the case   $G$ is a finite abelian subgroup of $\operatorname{SL}(n,\mathbb{C})$ and how to realize this variety as a toric variety. Applying techniques from toric geometry, we give a recipe for constructing minimal models of $\mathbb{C}^{n}/G$. Hilbert basis resolutions are also discussed. The main references for that section are \cite{Cox11,Dais06,Markushevich87,Yamagishi18,Yam23}.  

In Section $3$ we further  explore the structure of Hilbert basis resolutions and toric crepant resolutions of quotient singularities \(\mathbb{C}^n / G\). The results presented there extend the understanding of the exceptional sets of these resolutions and their relationship to junior elements of \( G \). We begin by proving a connection between the deformation retracts of the exceptional divisors corresponding to junior elements of $G$ and specific toric varieties associated with the fan of the resolution. More precisely, we show that if a junior element of \( G \) corresponds to a ray in the fan of a Hilbert basis resolution of $\mathbb{C}^{n}/G$, the associated irreducible component of the exceptional set is a deformation retract of an open toric subvariety (Theorem \ref{t3.5}). Furthermore, for \( n = 3 \), we identify compact junior elements of \( G \), and show that a toric crepant resolution of \(\mathbb{C}^3 / G\) is a toric glueing of the total spaces of the canonical bundles of the exceptional divisors associated with these compact junior elements (Corollary \ref{cor3.7}).

\medskip
\par \noindent{\bf Acknowledgment.} We thank Michele Graffeo and Dimitri Markushevich for a useful discussions and suggestions. 
The second author thanks SISSA and its Geometry and Mathematical Physics group for the warm hospitality and the pleasant working environment during his visit in the 2023-2024 academic year. 

\bigskip
\section{Gorenstein abelian quotient singularities}

Let $G$ be a finite abelian subgroup of $\operatorname{SL}(n,\mathbb{C})$. This section recalls the properties of the geometric quotient $\mathbb{C}^{n}/G$ and the situations in which it admits a crepant resolution.   Hilbert basis resolutions are also discussed. 

\subsection{$\mathbb{C}^{n}/G$ as a toric variety} Since $G\subset \operatorname{SL}(n,\mathbb{C})$ is abelian and finite, it is conjugated to a diagonal subgroup, so we may think it is diagonal.
Let be $r$ its  order. If $g\in G$, then $g=\operatorname{diag}(\lambda_{1},\dots,\lambda_{n})$ and $g^{r}=Id_{n\times n}$ 
so that $\lambda_{i}$ is a $r$-th root of $1$. Let $\epsilon_{r}$ be a fixed primitive $r$-th root of $1$. There are integers $0\leq a_{1},\dots, a_{n}\leq r-1$ such that $\lambda_{i}=\epsilon_{r}^{a_i}$. Since $\det(g)=1$, it follows that $a_{1}+\dots+a_{n}\equiv 0 \mod r$,  hence $\frac{1}{r}(a_{1}+\dots+a_{n})= m$  for some integer $m\in\{0,1,\dots, n-1\}$.
\begin{definition}
Let $g=\operatorname{diag}(\epsilon_{r}^{a_{1}}, \dots,\epsilon_{r}^{a_{n}})\in G$ as above.
\begin{itemize}
    \item[(a)] The \emph{age} of $g$  is  the number $\operatorname{age}(g):= \frac{1}{r}(a_{1}+\dots+a_{n})$. If $\operatorname{age}(g)=0$ then $g=Id_{n\times n}$. If $\operatorname{age}(g)= 1$, then g is   a \emph{junior element} of $G$. If $\operatorname{age}(g)>1$ then $g$ is   a \emph{senior element} of $G$. 
    \item[(b)] We denote by $\hat{g}$ the $n$-tuple $\frac{1}{r}(a_{1},\dots,a_{n})\in \mathbb{Q}^{n}$ and call it the \emph{fractional expression} of $g$. We denote by $\hat{G}$ the set $\{\hat{g}| g\in G\}\subset\mathbb{Q}^{n}$.
\end{itemize}
\end{definition} 

For $i\in\{1,\dots,n\}$, let $\eta_{i}:G\longrightarrow \mu_{r}$ be the projections, i.e, $$\eta_{i}(\operatorname{diag}(\epsilon_{r}^{a_{1}}, \dots,\epsilon_{r}^{a_{i}},\dots\epsilon_{r}^{a_{n}}))= \epsilon_{r}^{a_{i}}.$$ The characters $\eta_{i}$ define a monomorphism $\eta:G\longrightarrow \mu_{r}^{n}$,  $\eta(g)=(\eta_{1}(g),\dots,\eta_{n}(g))$. 
$G$ is a closed subgroup of $(\mathbb{C}^{*})^{n}$ by the the inclusions 
$$G\overset{\eta}{\longrightarrow}\mu_{r}^{n}\subset (\mathbb{C}^{*})^{n}; $$  
applying the functor $X(\_)$, which takes an algebraic group to the group of its characters,   
one gets a surjective group homomorphism  
$$X((\mathbb{C}^{*})^{n})\longrightarrow X(G).
$$    
One can identify $\mathbb{Z}^{n}$ with $X((\mathbb{C}^{*})^{n})$ by the isomorphism that takes an $n$-tuple of integers $(b_{1},\dots,b_{n})$ to the Laurent monomial $x_{1}^{b_{1}}\dots x_{n}^{b_{n}}$, as usual in toric geometry. In this way, from the      map above  we get the homomorphism
$\psi:\mathbb{Z}^{n}\twoheadrightarrow X(G)$ which takes an $n$-tuple $(b_{1},\dots,b_{n})$ to the character $b_{1}\eta_{1}+\dots+b_{n}\eta_{n}$. Letting $M$ be the kernel of this morphism, we have the exact sequence 
\begin{equation}
0 \longrightarrow M \stackrel{i}{\longrightarrow} \mathbb{Z}^n \stackrel{\psi}{\longrightarrow} X(G) \longrightarrow 0.\label{M}
\end{equation}

Applying the functor \( \operatorname{Hom}(-, \mathbb{Z}) \) to the above sequence and denoting \( N := \operatorname{Hom}(M, \mathbb{Z}) \), we see that \( N \) is a lattice that contains \( \mathbb{Z}^n \). Note that \( M \) is a sublattice of finite index in \( \mathbb{Z}^n \), and it is also the set of invariant monomials for the action of \( G \) on \( \mathbb{C}^n \).

\begin{proposition}
    With the notation above, one has $$N=\mathbb{Z}^{n}+\sum_{g\in G}\mathbb{Z}\hat{g} \quad\text{ and } \quad N/\mathbb{Z}^{n}\cong G.$$ 
\end{proposition}
\begin{proof}
   By definition, $(b_{1},\dots,b_{n})\in M$ if and only if $b_{1}\eta_{1}+\dots+b_{n}\eta_{n}=1$, and this is true if and only if for every element $g=\operatorname{diag}((\epsilon_{r}^{a_{1}}, \dots,\epsilon_{r}^{a_{n}})\in G$
   $$\sum_{i=1}^{n}b_{i}a_{i}\equiv 0 \mod r.$$
   Thus 
   \begin{eqnarray}
   M=\{m\in\mathbb{Z}^{n}| \langle m,\hat{g}\rangle\in\mathbb{Z} \text{ for every } \hat{g}\in\hat{G}\}.    
   \end{eqnarray}
 Let $N'$ denote $\mathbb{Z}^{n}+\sum_{g\in G}\mathbb{Z}\hat{g}$ and let $M'$ be the dual of $N'$. By the above description of $M$, we have $\mathbb{Z}^{n}\subset N'\subset N$ and $M\subset M'\subset \mathbb{Z}^{n}$. Note that every element of $M'$ is in the kernel $M$ of $\psi$. Therefore, $M=M'$ and $N=N'$. In particular, $N/\mathbb{Z}^{n}\cong G$ where the isomorphism assigns an element $g=\operatorname{diag}(\epsilon^{a_{1}}, \dots, \epsilon^{a_{n}})\in G$ to the class of its fractional expression $\hat{g}=\frac{1}{r}(a_{1},\dots,a_{n})$.    
\end{proof}

Let $\{e_{1},\dots, e_{n}\}$ be the standard basis of $\mathbb{R}^{n}$ and $\sigma:=\operatorname{Cone}(e_{1},\dots,e_{n})$.
The fact that $N_{\mathbb{R}}=\mathbb{R}^{n}$ implies the existence of a toric variety $U_{\sigma, N}$ with torus $T_{N}$. In the next proposition, we check that $\mathbb{C}^{n}/G$ is, in fact, $U_{\sigma,N}$.
\begin{proposition} $\mathbb{C}^{n}/G=U_{
    \sigma, N}=\operatorname{Specm}(\mathbb{C}[\sigma^{\vee}\cap M])$. 
\end{proposition}
\begin{proof}
Using the theory of homogeneous coordinates for toric varieties \cite{Cox95} we   express $U_{\sigma,N}$ in homogeneous coordinates. Let $(\Sigma,N)$ be the fan that consists of $\sigma$ and its faces. Thus, $X_{\Sigma,N}=U_{\sigma,N}$ is a simplicial toric variety with no torus factors, and in this case $X_\Sigma$ is a geometric quotient $(\mathbb{C}^{\Sigma(1)} \setminus Z(\Sigma)) / G'$, where $G'=\operatorname{Hom}_\mathbb{Z}(\operatorname{Cl}(X_\Sigma),\mathbb{C}^{*})\subset(\mathbb{C}^{*})^{n}$ and $Z(\Sigma)$ is the set of zeros of the irrelevant ideal $B(\Sigma)$ of $\Sigma$. 
The total coordinate ring of $X_{\Sigma,N}$ is the polynomial ring $\mathbb{C}[x_{1},\dots,x_{n}]$,  where $x_{i}$ is the variable corresponding to the ray $\rho_{i}$. The unique maximal cone of $\Sigma$   is $\sigma$, hence the irrelevant ideal is   $B(\Sigma)=\langle x_{\hat{\sigma}}\rangle$ where 
\[
x_{\widehat{\sigma}} = \prod_{\rho \not\in \sigma(1)} x_\rho.
\] However, since $\Sigma(1)$ consists only of rays of $\sigma$, $x_{\hat{\sigma}}=1$. In this way, $Z(\Sigma)=\emptyset$ and then $X_{\Sigma}=\mathbb{C}^{n}/G'$. The group $G'$ consists of the points $(x_{1},\dots,x_{n})$ of $(\mathbb{C}^{*})^{n}$ such that, for every $m=(b_{1},\dots b_{n})\in M$, $$1=x_{1}^{\langle
e_{1},m\rangle} \dots x_{1}^{\langle
e_{1},m\rangle}=x_{1}^{b_{1}}\dots x_{n}^{b_{n}}.$$ 
all of the elements of $G$ satisfy this equation, hence $G\subset G'$. Since $X_{\Sigma,N}$ has not torus factors, the sequence 
\begin{eqnarray*}
  0 \longrightarrow M \longrightarrow \operatorname{Div}_{T_{N}}(X_{\Sigma,N}) \longrightarrow \operatorname{Cl}(X_{\Sigma,N}) \longrightarrow 0  
\end{eqnarray*}
is exact,
so that by \eqref{M}  $X(G)\cong\operatorname{Cl}(X_{\Sigma,N})$. Since a finite group has the same order of its group of characters, we have $$|G|=|X(G)|=|\operatorname{Cl}(X_{\Sigma,N})|=|G'|.$$
Hence $G=G'$ and therefore $U_{\sigma,N}=\mathbb{C}^{n}/G$. 
\end{proof}

Thus $\mathbb{C}[x_{1},\dots,x_{n}]^{G}=\mathbb{C}[\sigma^{\vee}\cap M]$. In particular, the canonical morphism $\pi:\mathbb{C}^{n}\rightarrow \mathbb{C}^{n}/G=U_{\sigma,N}$ is the toric morphism coming from the inclusion $\mathbb{Z}^{n}\subset N$. Moreover $\pi((\mathbb{C}^{*})^{n})=(\mathbb{C}^{*})^{n}/G=T_{N}$, and if $\tau$ is a face of $\sigma$ then $\pi(O_{\mathbb{Z}^{n}}(\tau))=O_{\mathbb{Z}^{n}}(\tau)/G=O_{N}(\tau)$ is the orbit in $U_{\sigma,N}$ corresponding to $\tau$. Another way to see this is to check that the action of $G$ on $\mathbb{C}^{n}$ commutes with the action of the torus $(\mathbb{C}^{*})^{n}$. 

 \begin{remark}\label{Qif}
  The pair $(\sigma,\mathbb{Z}^{n})$ is a smooth cone, but the pair $(\sigma,N)$ is no longer smooth, as it corresponds to the quotient variety $\mathbb{C}^{n}/G$. However, the ray generators $e_{1},\dots,e_{n}$ of $\sigma$ remain linearly independent in $N_{\mathbb{R}}=\mathbb{R}^{n}$, and hence $(\sigma,N)$ is simplicial. Thus, $\mathbb{C}^{n}/G$ is a $\mathbb{Q}$-factorial variety.   
 \end{remark}
\begin{remark}
The divisor $K_{U_{\sigma,N}}=-(D_{\rho_{e_{1}}}+\dots+D_{\rho_{e_{n}}})$, where $\rho_{e_{i}}=\operatorname{Cone}(e_{i})$, 
is a canonical divisor of $\mathbb{C}^{n}/G=U_{\sigma,N}$. Since the monomial $x_{1}\dots x_{n}$ is $G$-invariant, $m=(1,\dots,1)\in \sigma^{\vee}\cap M$. The support function of $K_{U_{\sigma,N}}$ is given by $\varphi_{K_{U_{\sigma,N}}}(u)=\langle m,u\rangle$ for every $u\in\sigma$. In particular, $K_{U_{\sigma,N}}$ is a principal divisor, and this also shows that $\mathbb{C}^{n}/G$ is, in fact, a Gorenstein variety.  
\end{remark}

The set $$ \square:=\{(\alpha_{1},\dots,\alpha_{n})\in \mathbb{R}^{n}| 0\leq\alpha_{i}<1\}$$ is a fundamental domain for the action of $\mathbb{Z}^{n}$ on $\mathbb{R}^{n}$, so that  every element of $G$ has a unique representative in $\square$, i.e, $N\cap\square=\{\hat{g}|g\in G\}:=\hat{G}$. 
In particular,  the  unique representations
of  the junior classes of $G$ are  contained in the slice 
$$\{(\alpha_{1},\dots,\alpha_{n})\in\square\,\bigm\vert\, \alpha_{1}+\dots+\alpha_{n}=1\}.$$

 Let $H_{m,1}$ be affine hyperplane     $\langle m,x\rangle=1$. Note that
    $$H_{m,1}\cap \sigma=\operatorname{Conv}(e_{1},\dots, e_{n}):=\bigtriangleup$$
and the junior elements of $G$ have their fractional expressions contained in $\bigtriangleup$. Moreover, $$N\cap\bigtriangleup=\{\hat{g}\in \hat{G}|\operatorname{age}(g)=1\}\cup\{e_{1},\dots,e_{n}\}:=\nu_{G}.$$
The set $\nu_{G}$ is called the \emph{junior simplex} of $G$ and the picture of the points of $\nu_{G}$ in $\bigtriangleup$ is called  the \emph{graph} of $G$.

\vspace{0.3cm}
Before giving an example that illustrates the discussion above, it is important to define an   interesting kind of Gorenstein quotient singularity,  those that admit the canonical Fujiki-Oka resolutions (see \cite{Sato21}). This singles out a class of quotient singularities that are easier to work with. 

\begin{definition}
  An element $\frac{1}{r}(a_{1},\dots,a_{n})\in\hat{G}$ is called \emph{semi-unimodular} if at least one of the $a_{i}$'s  is $1$. Suppose $G$ is a cyclic group generated by an element $g$ corresponding to a semi-unimodular element of $\hat{G}$. Up to a change of coordinates, that element can be written as $\hat{g}=\frac{1}{r}(1,a_{1},\dots,a_{n-1})$. In this case, $G$ is denoted by $\mathbb{Z}_{r,(1,a_{1},\dots,a_{n-1})}$ and  $\mathbb{C}^{n}/G$ is said to be a \emph{cyclic quotient singularity} of $\frac{1}{r}(1,a_{1},\dots,a_{n-1})-$type.   
\end{definition}
\begin{example}\label{Z_6,1}
 Denote by $\mathbb{Z}_{6,(1,2,3)}$  the cyclic subgroup of $\operatorname{SL}(3,\mathbb{C})$ generated by the diagonal matrix
 
 \[g_{1}=
\begin{bmatrix}
 \epsilon_{6} & 0 & 0 \\
0 & \epsilon_{6}^{2} & 0 \\
0 & 0 & \epsilon_{6}^{3}
\end{bmatrix}
.\]
$\mathbb{C}^{3}/\mathbb{Z}_{6,(1,2,3)}$ is the toric variety $U_{\sigma,N}$ where $N=\mathbb{Z}^{3}+\mathbb{Z}\frac{1}{6}(1,2,3)$ and $\sigma=\operatorname{Cone}(e_{1},e_{2},e_{3})$. The junior elements of this group are $g_{i}:=g_{1}^{i}$ for $i\in\{1,2,3,4\}$ and its unique senior element is $g_{5}:=g_{1}^{5}$. Their representations $\hat{g_{1}}=\frac{1}{6}(1,2,3),\hat{g_{2}}=\frac{1}{6}(2,4,0),\hat{g_{3}}=\frac{1}{6}(3,0,3),\hat{g_{4}}=\frac{1}{6}(4,2,0)$ in $\bigtriangleup$ provide the    graph of $\mathbb{Z}_{6,(1,2,3)}$:
 \begin{center}
\begin{tikzpicture}[scale=0.50]
 	\draw (0,0) -- (8,0); 
 	\draw (0,0) -- (4,6.9); 
 	\draw (8,0) -- (4,6.9);   

 	\draw [fill] (0,0) circle (3pt); 
 	\draw [fill] (8,0) circle (3pt); 
 	\draw [fill] (4,6.9) circle (3pt); 

 	\draw [fill] (4.3,3.45) circle (3pt); 
 	\draw [fill] (5.33,0) circle (3pt); 
 	\draw [fill] (2,3.45) circle (3pt); 
 	\draw [fill] (2.67,0) circle (3pt); 

 	\node at (-0.5,0) {$\mathbf e_1$}; 
 	\node at (8.6,0) {$\mathbf e_2$}; 
 	\node at (4.2,7.4) {$\mathbf e_3$}; 
 	
 	\node at (4.9,3.45) {$\small\hat{g_{1}}$};
 	\node at (2.67,-0.6) {$\hat{g_{4}}$}; 
 	\node at (5.33,-0.6) {$\hat{g_{2}}$};
 	\node at (1.4,3.45) {$\hat{g_{3}}$};

\end{tikzpicture}
\end{center}
The singular faces of $\sigma$ with respect to $N$ are $\tau_{1}=\operatorname{Cone}(e_{1},e_{2}), \tau_{2}=\operatorname{Cone}(e_{1},e_{3})$ and $\sigma$ itself. This way, the set o singular points of $\mathbb{C}^{3}/\mathbb{Z}_{6,(1,2,3)}$ is given by  $$(\mathbb{C}^{3}/\mathbb{Z}_{6,(1,2,3)})_{\operatorname{sing}}=V_{N}(\tau_{1})\cup V_{N}(\tau_{2})\cup V_{N}(\sigma)= \pi(C_{1})\cup\pi(C_{2})\cup\pi(\{ 0 \})$$
 where $V_{N}(\tau_{i})$ denotes closure of the toric orbit $O_{N}(\tau_{i})$, while  $C_{1}=\{(x,y,z)\in\mathbb{C}^{3}|x=z=0\}$, $C_{2}=\{(x,y,z)\in\mathbb{C}^{3}|x=y=0\}$, and $\pi:\mathbb{C}^{3}\rightarrow \mathbb{C}^{3}/\mathbb{Z}_{6,(1,2,3)}$ is the canonical morphism.
  \end{example}
  
Note that in the previous example the singular faces of $\sigma$ are those that contain at least a fractional expression of a junior element of $\mathbb{Z}_{6,(1,2,3)}$ in their relative interior. In the next section, we will see that this is not a coincidence.

\subsection{Crepant resolutions} We start this section by  recalling the definition of minimal model for the quotient variety $\mathbb{C}^{n}/G$.
\begin{definition}
    A \emph{minimal model} of \(\mathbb{C}^{n}/G\) is a \(\mathbb{Q}\)-factorial normal variety \(X\) which
has only terminal singularities together with a crepant proper birational morphism 
  $X \;\longrightarrow\; \mathbb{C}^{n}/G.$
\end{definition}

In \cite[Sec.~3.3]{Yam23} it was proved that any minimal model of $\mathbb{C}^{n}/G$ is toric. In particular, if $\mathbb{C}^{n}/G$ admits a crepant resolution $X\to\mathbb{C}^{n}/G$, then $X$ should be toric, because any smooth variety has terminal singularities (see \cite[p.~107]{Ishii18}). In this way, following the steps described in \cite{Dais06} and in the Appendix of \cite{Markushevich87}, we can give the recipe to get (toric) minimal models for $\mathbb{C}^{n}/G$. In dimensions $2$ and $3$ any minimal model for such quotient variety is a crepant resolution.

We say that a fan  $(\Sigma',L)$ refines (or subdivides) a fan  $(\Sigma,L)$ 
 if  $\lvert \Sigma' \rvert = \lvert\Sigma \rvert$ 
 and every cone of  $\Sigma'$  is contained in a cone of  $\Sigma$. In this situation, the identity map  $\mathrm{id} \colon L \longrightarrow L$ 
induces a toric morphism $\phi: X_{\Sigma'} \longrightarrow X_{\Sigma}$, which is proper  and is birational since it is the identity on $T_{L}$. There is a way to construct  a nice refinements of a fan, which is defined as follows. 
\begin{definition}
    Let $\Xi$ be a fan in $L_{\mathbb{R}}$. Given a primitive element $\mu \in |\Xi|\cap L$, the \emph{star subdivision} of $\Xi$ at $\mu$ is the fan denoted by $\Xi^*(\mu)$ and defined  as the set containing the following cones:
    \begin{itemize}
        \item[$(a)$]  $\xi\in\Xi$, where $\mu\not\in \xi$;
        \item[$(b)$]  $\operatorname{Cone}(\tau,\mu)$, where $\mu\not\in\tau\in\Xi$ and $\{\mu\}\cup\tau\subset \xi\in\Xi$. 
    \end{itemize}
    \end{definition}
    The main properties of the star subdivision are the following:
    \begin{itemize}
        \item the $1$-dimensional cones of $\Xi^*(\mu)$ are the 1-dimensional cones of $\Xi$, plus the cone generated by $\mu$;
        \item $\Xi^*(\mu)$ is a refinement of $\Xi$;
        \item the corresponding toric morphism $X_{\Xi^*(\mu)}\rightarrow X_{\Xi}$ is projective;
        \item if $\Xi$ is simplicial then $\Xi^*(\mu)$ is simplicial as well.
    \end{itemize}

\begin{lemma}
Suppose that $X_{\Sigma}$ is a Gorenstein toric variety. Let $\varphi$ be the support function of the Cartier divisor $K_{X_\Sigma}$, 
and let $\phi \colon X_{\Sigma'} \longrightarrow X_{\Sigma}$ be the toric morphism 
coming from a refinement $\Sigma'$ of $\Sigma$. Then one has the toric ramification formula
\[
  K_{X_{\Sigma'}} 
  \;=\; 
  \phi^*K_{X_\Sigma} 
    \;+\;
    \sum_{\rho \in \Sigma'(1)\setminus\Sigma(1)} 
      \bigl(\varphi(u_\rho) - 1 \bigr)\,D_\rho.
\]
\end{lemma}

\begin{proof}
Note that that $K_{X_\Sigma}$ and its pullback $\phi^*K_{X_\Sigma}$ 
have the same support function. Thus
\[
  \phi^*K_{X_\Sigma}
  \;=\; 
  -\sum_{\rho \in \Sigma'(1)} \varphi(u_\rho)\,D_\rho.
\]
The required formula now follows at once, relying on the fact that 
$\varphi(u_\rho) = 1$ for each $\rho \in \Sigma(1)$.
\end{proof}

\medskip

\begin{proposition}
If  $X_{\Sigma,L}$ is a normal Gorenstein toric variety, $X_\Sigma$ has 
canonical singularities.
\end{proposition}

\begin{proof}
Observe first that the support function \(\varphi\) associated with \(K_{X_\Sigma}\) 
is integral along \(L\) by virtue of \(X_\Sigma\) being Gorenstein.  
Next, choose a smooth refinement \(\Sigma'\) of \(\Sigma\).  
Then for each \(u_\rho \in \Sigma'(1)\), the following properties hold:
\begin{enumerate}
\item 
\(\varphi(u_\rho) \in \mathbb{Z}\), since \(u_\rho\) lies in \(\lvert\Sigma\rvert \cap N\).
\item 
\(\varphi(u_\rho) > 0\), because \(u_\rho\) lies in some cone \(\sigma \in \Sigma\), 
and \(\varphi\) takes the value \(1\) on each minimal generator of \(\sigma\). 
It follows that \(\varphi\) remains strictly positive on \(\sigma\).
\end{enumerate}
Hence \(\varphi(u_\rho) \ge 1\).  By the ramification formula for toric morphisms, we conclude that \(X_\Sigma\) must have canonical singularities.
\end{proof}

 Bearing this result in mind, it follows that $\mathbb{C}^{n}/G$ has canonical singularities since it is a Gorenstein toric variety. The next proposition gives us the ideia of which points we should choose to refine the fan of $\mathbb{C}^{n}/G$ to get a minimal model for it.

\begin{proposition}    \cite[Prop.~11.4.12]{Cox11} \label{p.c.T}
Let $L$ be a lattice and let $\xi\subset L_{\mathbb{R}}$ be a simplicial strongly convex rational polyhedral cone with ray generators $u_{1},\dots u_{d}$. Define $$\Psi_{\xi}=\operatorname{Conv}(0,u_{i}|i=1,\dots d).$$
 	If $U_{\xi}$ is $\mathbb{Q}$-Gorenstein, then
    \begin{itemize}
 		\item[(i)] $\Psi_{\xi}$ has a unique facet not containing the origin.
 		\item[(ii)] $U_{\xi,L}$ has terminal singularities if and only if the only lattice points of $\Psi_{\xi}$ are given by its vertices.
 		\item[(iii)] $U_{\xi,L}$ has canonical singularities if and only if the only nonzero lattice points of  $\Psi_{\xi}$ lie in the facet not containing the origin.
 	\end{itemize} 
 \end{proposition}
  
 In particular, if $U_{\xi,L}$ is Goreinstein and $n$-dimensional, and its fan has   rays $u_{1},\dots,u_{n}$, with $rank(L)=n$,  then the face of $\Psi_{\xi}$ that does not contain the origin is 
\begin{eqnarray}\label{Ttau}
	T_{\xi}:=\operatorname{Conv}(u_{1},\dots,u_{n}),
\end{eqnarray}
which is a triangle that lies in the hyperplane defined by $m\in M$ such that $K_{U_{\xi}}=\operatorname{Div}(\chi^{m})$.
 Applying this proposition to our case, we get $\Psi_{\sigma}=\operatorname{Conv}(0,e_{1},\dots,e_{n})$ and the facet not containing the origin is   $\bigtriangleup=\operatorname{Conv}(e_{1},\dots, e_{n})$. Since $\mathbb{C}^{n}/G= U_{\sigma, N}$ is Gorenstein and has canonical singularities, it follows the lattice points of $\Psi_{\sigma}$ are given by the set 
\begin{eqnarray*}
\nu_{G}=\{\hat{g}\in \hat{G}|\operatorname{age}(g)=1\}\cup\{e_{1},\dots,e_{n}\}.
\end{eqnarray*}

 Now we see how to find the toric minimal models of $\mathbb{C}^{n}/G=U_{\sigma,N}$. For every junior class $g$ of $G$, $\varphi_{K_{U_{\sigma,N}}}(\hat{g})=1$, by Proposition \ref{p.c.T}  any simplicial refinement of $\sigma$ whose new rays  are those   generated by the element of  $\nu_{G}\cap\hat{G}$ provides a toric crepant morphism $\phi:X_{\Sigma}\rightarrow\mathbb{C}^{n}/G$, because of the toric ramification formula 
 \[
  K_{X_{\Sigma}} 
  \;=\; 
  \phi^*(K_{U_{\sigma,N}}) 
    \;+\;
    \sum_{\hat{g}\in \nu_{G}\cap\hat{G}} 
      \bigl(\varphi_{K_{U_{\sigma,N}}}(\hat{g}) - 1 \bigr)\,D_{Cone(\hat{g})}.
\]
In this situation, $X_{\Sigma}$ has terminal singularities.
\begin{remark}\label{R.S.T}
  When $n=2,3$ this  morphism is a crepant resolution of singularities because every Gorenstein toric surface and every Gorenstein simplicial $3$-dimensional toric variety with terminal singularities is smooth  \cite[Prop.~11.4.19]{Cox11}.
\end{remark}
The easiest way to obtain projective morphisms of  this type  (i.e., a minimal model) is to do a sequence of star subdivisions of $\sigma$ at each element of $\nu_{G}\cap\hat{G}$. Performing a star subdivision of the fan that consists of $\sigma$ and its faces provides us with an $n$-dimensional graph. However, when such a star subdivision is done at a point $\hat{g}\in\nu_{G}$, the graph becomes the 2-dimensional graph of $G$. It is reduced to a triangulation of $\bigtriangleup$ with a new vertex $\hat{g}$, which we   denote by $\bigtriangleup_{\hat{g}}$ and, by abuse of language,  call a star subdivision of $\bigtriangleup$ at $\hat{g}$. The fan associated to $\bigtriangleup_{\hat{g}}$ is denoted by $\Sigma^{*}(\hat{g})$. Another star subdivision of $\bigtriangleup_{\hat{g}}$ at a point $\hat{h}\in\nu\symbol{92}\{\hat{g}\}$ will be denoted by $\bigtriangleup_{(\hat{g},\hat{h})}$ (and its associated fan by $\Sigma^{*}(\hat{g},\hat{h})$) and so on. Note that it can happen $\bigtriangleup_{(\hat{g},\hat{h})}\neq \bigtriangleup_{(\hat{h},\hat{g})}$, as it is shown in the next example.
 \begin{example}\label{Z_{6}}
  Following the notation of Example \ref{Z_6,1}, we will find two projective crepant resolutions of $\mathbb{C}^{3}/\mathbb{Z}_{6,(1,2,3)}$ by performing different sequences of star subdivisions of $\bigtriangleup$. At first, we take
  the sequence of star subdivisions  $(\hat{g_{1}},\hat{g_{2}},\hat{g_{3}},\hat{g_{4}})$, and the sequence of   corresponding graphs is
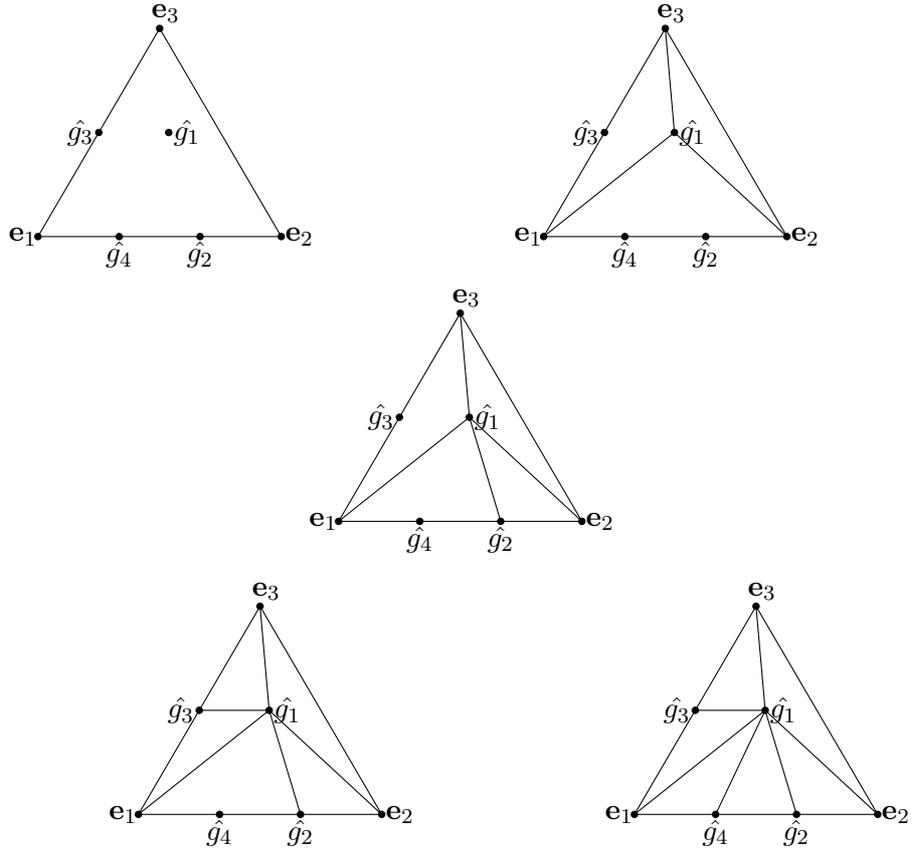
\begin{center}
\begin{tikzpicture}[scale=0.40]
 	\draw (0,0) -- (8,0); 
 	\draw (0,0) -- (4,6.9); 
 	\draw (8,0) -- (4,6.9);   

 	\draw [fill] (0,0) circle (3pt); 
 	\draw [fill] (8,0) circle (3pt); 
 	\draw [fill] (4,6.9) circle (3pt); 

 	\draw [fill] (4.3,3.45) circle (3pt); 
 	\draw [fill] (5.33,0) circle (3pt); 
 	\draw [fill] (2,3.45) circle (3pt); 
 	\draw [fill] (2.67,0) circle (3pt); 

 	\node at (-0.5,0) {$\mathbf e_1$}; 
 	\node at (8.6,0) {$\mathbf e_2$}; 
 	\node at (4.2,7.4) {$\mathbf e_3$}; 
 	
 	\node at (4.9,3.45) {$\small\hat{g_{1}}$};
 	\node at (2.67,-0.6) {$\hat{g_{4}}$}; 
 	\node at (5.33,-0.6) {$\hat{g_{2}}$};
 	\node at (1.4,3.45) {$\hat{g_{3}}$};
\end{tikzpicture} \hskip 1cm 
 \raisebox{5em}{
 \begin{tikzpicture}
     $\xrightarrow{(\hat{g_{1}})}$
 \end{tikzpicture}} \hskip1cm
 \begin{tikzpicture}[scale=0.40]
 	\draw (0,0) -- (8,0); 
 	\draw (0,0) -- (4,6.9); 
 	\draw (8,0) -- (4,6.9);   
\draw (4.3,3.45) -- (0,0);
\draw (4.3,3.45) -- (4,6.9);
\draw (4.3,3.45) -- (8,0);
 	\draw [fill] (0,0) circle (3pt); 
 	\draw [fill] (8,0) circle (3pt); 
 	\draw [fill] (4,6.9) circle (3pt); 

 	\draw [fill] (4.3,3.45) circle (3pt); 
 	\draw [fill] (5.33,0) circle (3pt); 
 	\draw [fill] (2,3.45) circle (3pt); 
 	\draw [fill] (2.67,0) circle (3pt); 

 	\node at (-0.5,0) {$\mathbf e_1$}; 
 	\node at (8.6,0) {$\mathbf e_2$}; 
 	\node at (4.2,7.4) {$\mathbf e_3$}; 
 	
 	\node at (4.9,3.45) {$\small\hat{g_{1}}$};
 	\node at (2.67,-0.6) {$\hat{g_{4}}$}; 
 	\node at (5.33,-0.6) {$\hat{g_{2}}$};
 	\node at (1.4,3.45) {$\hat{g_{3}}$};
 \end{tikzpicture} \hskip1cm
 \raisebox{5em}{
 \begin{tikzpicture}
     $\xrightarrow{(\hat{g_{1}},\hat{g_{2}})}$
 \end{tikzpicture}} \hskip 1cm
 \begin{tikzpicture}[scale=0.40]
 	\draw (0,0) -- (8,0); 
 	\draw (0,0) -- (4,6.9); 
 	\draw (8,0) -- (4,6.9);   
\draw (4.3,3.45) -- (0,0);
\draw (4.3,3.45) -- (4,6.9);
\draw (4.3,3.45) -- (8,0);
\draw (4.3,3.45) -- (5.33,0);
 	\draw [fill] (0,0) circle (3pt); 
 	\draw [fill] (8,0) circle (3pt); 
 	\draw [fill] (4,6.9) circle (3pt); 

 	\draw [fill] (4.3,3.45) circle (3pt); 
 	\draw [fill] (5.33,0) circle (3pt); 
 	\draw [fill] (2,3.45) circle (3pt); 
 	\draw [fill] (2.67,0) circle (3pt); 

 	\node at (-0.5,0) {$\mathbf e_1$}; 
 	\node at (8.6,0) {$\mathbf e_2$}; 
 	\node at (4.2,7.4) {$\mathbf e_3$}; 
 	
 	\node at (4.9,3.45) {$\small\hat{g_{1}}$};
 	\node at (2.67,-0.6) {$\hat{g_{4}}$}; 
 	\node at (5.33,-0.6) {$\hat{g_{2}}$};
 	\node at (1.4,3.45) {$\hat{g_{3}}$};

 \end{tikzpicture}
 
\raisebox{5em}{
 \begin{tikzpicture}
     $\xrightarrow{(\hat{g_{1}},\hat{g_{2}},\hat{g_{3}})}$
 \end{tikzpicture}} \hskip 1cm
 \begin{tikzpicture}[scale=0.40]
 	\draw (0,0) -- (8,0); 
 	\draw (0,0) -- (4,6.9); 
 	\draw (8,0) -- (4,6.9);   
\draw (4.3,3.45) -- (0,0);
\draw (4.3,3.45) -- (4,6.9);
\draw (4.3,3.45) -- (8,0);
\draw (4.3,3.45) -- (5.33,0);
\draw (4.3,3.45)-- (2, 3.45);
 	\draw [fill] (0,0) circle (3pt); 
 	\draw [fill] (8,0) circle (3pt); 
 	\draw [fill] (4,6.9) circle (3pt); 

 	\draw [fill] (4.3,3.45) circle (3pt); 
 	\draw [fill] (5.33,0) circle (3pt); 
 	\draw [fill] (2,3.45) circle (3pt); 
 	\draw [fill] (2.67,0) circle (3pt); 

 	\node at (-0.5,0) {$\mathbf e_1$}; 
 	\node at (8.6,0) {$\mathbf e_2$}; 
 	\node at (4.2,7.4) {$\mathbf e_3$}; 
 	
 	\node at (4.9,3.45) {$\small\hat{g_{1}}$};
 	\node at (2.67,-0.6) {$\hat{g_{4}}$}; 
 	\node at (5.33,-0.6) {$\hat{g_{2}}$};
 	\node at (1.4,3.45) {$\hat{g_{3}}$};

 \end{tikzpicture}\hskip1cm 
 \raisebox{5em}{
 \begin{tikzpicture}
     $\xrightarrow{(\hat{g_{1}},\hat{g_{2}},\hat{g_{3}}, \hat{g_{4}})}$
 \end{tikzpicture}} \hskip1cm
 \begin{tikzpicture}[scale=0.40]
 	\draw (0,0) -- (8,0); 
 	\draw (0,0) -- (4,6.9); 
 	\draw (8,0) -- (4,6.9);   
\draw (4.3,3.45) -- (0,0);
\draw (4.3,3.45) -- (4,6.9);
\draw (4.3,3.45) -- (8,0);
\draw (4.3,3.45) -- (5.33,0);
\draw (4.3,3.45)-- (2, 3.45);
\draw (4.3,3.45)-- (2.67, 0);
 	\draw [fill] (0,0) circle (3pt); 
 	\draw [fill] (8,0) circle (3pt); 
 	\draw [fill] (4,6.9) circle (3pt); 

 	\draw [fill] (4.3,3.45) circle (3pt); 
 	\draw [fill] (5.33,0) circle (3pt); 
 	\draw [fill] (2,3.45) circle (3pt); 
 	\draw [fill] (2.67,0) circle (3pt); 

 	\node at (-0.5,0) {$\mathbf e_1$}; 
 	\node at (8.6,0) {$\mathbf e_2$}; 
 	\node at (4.2,7.4) {$\mathbf e_3$}; 
 	
 	\node at (4.9,3.45) {$\small\hat{g_{1}}$};
 	\node at (2.67,-0.6) {$\hat{g_{4}}$}; 
 	\node at (5.33,-0.6) {$\hat{g_{2}}$};
 	\node at (1.4,3.45) {$\hat{g_{3}}$};
\end{tikzpicture}
 \end{center}
The last graph $\bigtriangleup_{(\hat{g_{1}},\hat{g_{2}},\hat{g_{3}},\hat{g_{4}})}$   corresponds to the fan $\Sigma^{*}(\hat{g_{1}},\hat{g_{2}},\hat{g_{3}},\hat{g_{4}})$. This is a refinement of $\sigma$;  the corresponding toric morphism $\phi_{(\hat{g_{1}},\hat{g_{2}},\hat{g_{3}},\hat{g_{4}})}:X_{\Sigma^{*}(\hat{g_{1}},\hat{g_{2}},\hat{g_{3}},\hat{g_{4}})}\longrightarrow\mathbb{C}^{3}/\mathbb{Z}_{6,(1,2,3)}$ is a minimal model and, therefore  by Remark \ref{R.S.T}  a crepant resolution for the quotient singularity. On the other hand, after performing the sequence $(\hat{g_{4}},\hat{g_{3}},\hat{g_{1}},\hat{g_{2}})$ of star subdivisions, we get $\bigtriangleup_{(\hat{g_{4}},\hat{g_{3}},\hat{g_{1}},\hat{g_{2}})}$, whose graph is
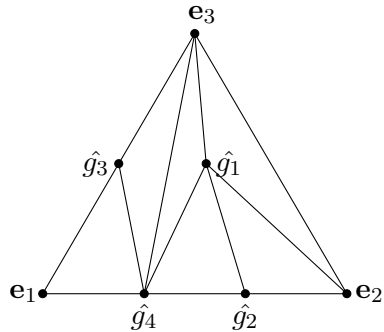
\begin{center}
 \begin{tikzpicture}[scale=0.50]
 	\draw (0,0) -- (8,0); 
 	\draw (0,0) -- (4,6.9); 
 	\draw (8,0) -- (4,6.9);   
\draw (2.67,0)-- (4,6.9);
\draw (2.67,0)-- (2,3.45);
\draw (2.67,0)-- (4.3,3.45);
\draw (4.3,3.45) -- (8,0);
\draw (4.3,3.45) -- (4,6.9);
\draw (4.3,3.45) -- (5.33,0);

 	\draw [fill] (0,0) circle (3pt); 
 	\draw [fill] (8,0) circle (3pt); 
 	\draw [fill] (4,6.9) circle (3pt); 

 	\draw [fill] (4.3,3.45) circle (3pt); 
 	\draw [fill] (5.33,0) circle (3pt); 
 	\draw [fill] (2,3.45) circle (3pt); 
 	\draw [fill] (2.67,0) circle (3pt); 

 	\node at (-0.5,0) {$\mathbf e_1$}; 
 	\node at (8.6,0) {$\mathbf e_2$}; 
 	\node at (4.2,7.4) {$\mathbf e_3$}; 
 	
 	\node at (4.9,3.45) {$\small\hat{g_{1}}$};
 	\node at (2.67,-0.6) {$\hat{g_{4}}$}; 
 	\node at (5.33,-0.6) {$\hat{g_{2}}$};
 	\node at (1.4,3.45) {$\hat{g_{3}}$};
\end{tikzpicture}
 \end{center}
 and the corresponding fan is  $\Sigma^{*}(\hat{g_{4}},\hat{g_{3}},\hat{g_{1}},\hat{g_{1}})$. This fan provides another toric minimal model $X_{\Sigma^{*}(\hat{g_{4}},\hat{g_{3}},\hat{g_{1}},\hat{g_{2}})}\longrightarrow\mathbb{C}^{3}/\mathbb{Z}_{6,(1,2,3)}$ for the singularity. 
  \end{example}
  
Now  two remarks are in order. First, not all the toric minimal models of the singularity $\mathbb{C}^{n}/G$ have a fan which is a sequence of star subdivisions at the points of $\hat{G}\cap\nu_{G}$. For instance, the fan related to the graph 
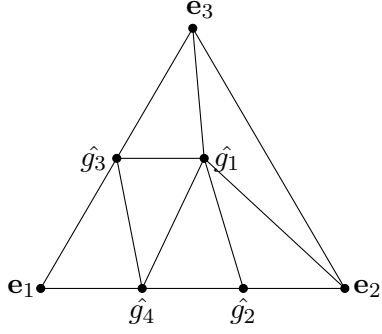
\begin{center}
\begin{tikzpicture}[scale=0.50]
 	\draw (0,0) -- (8,0); 
 	\draw (0,0) -- (4,6.9); 
 	\draw (8,0) -- (4,6.9);   
\draw (2,3.45) -- (2.67,0);
\draw (4.3,3.45) -- (4,6.9);
\draw (4.3,3.45) -- (8,0);
\draw (4.3,3.45) -- (5.33,0);
\draw (4.3,3.45)-- (2, 3.45);
\draw (4.3,3.45)-- (2.67, 0);
 	\draw [fill] (0,0) circle (3pt); 
 	\draw [fill] (8,0) circle (3pt); 
 	\draw [fill] (4,6.9) circle (3pt); 

 	\draw [fill] (4.3,3.45) circle (3pt); 
 	\draw [fill] (5.33,0) circle (3pt); 
 	\draw [fill] (2,3.45) circle (3pt); 
 	\draw [fill] (2.67,0) circle (3pt); 

 	\node at (-0.5,0) {$\mathbf e_1$}; 
 	\node at (8.6,0) {$\mathbf e_2$}; 
 	\node at (4.2,7.4) {$\mathbf e_3$}; 
 	
 	\node at (4.9,3.45) {$\small\hat{g_{1}}$};
 	\node at (2.67,-0.6) {$\hat{g_{4}}$}; 
 	\node at (5.33,-0.6) {$\hat{g_{2}}$};
 	\node at (1.4,3.45) {$\hat{g_{3}}$};

 \end{tikzpicture}

\end{center}
provides another minimal model for $\mathbb{C}^{3}/\mathbb{Z}_{6,(1,2,3)}$. However, both fan and graph cannot be obtained by a sequence of star subdivisions of the set $\{\hat{g_{1}},\hat{g_{2}},\hat{g_{3}},\hat{g_{4}}\}$. Second, suppose $\Xi$ is the fan of a crepant morphism of $\mathbb{C}^{n}/G$. In that case, the number of elements of $\Xi(m) $  equals the number of $m$-dimensional faces corresponding to the triangulation of $\bigtriangleup$, for $m> 0$.

\begin{remark}
If $\mathbb{C}^{n}/G=U_{\sigma,N}$ admits a toric crepant resolution $\phi:X_{\Sigma}\rightarrow \mathbb{C}^{n}/G$ coming from a simplicial refinement of $\sigma$ such that new rays added are those generated by each element of the set $\nu_{G}\cap\hat{G}$, then all the irreducible components of $\operatorname{Exc}(\phi)$ have codimension $1$ in $X_{\Sigma}$. Moreover, one can check that $$\operatorname{Exc}(\phi)=\bigcup_{\hat{g}\in\nu_{G}\cap\hat{G}} E_{g},$$ where $E_{g}=V(Cone(\hat{g}))\subset X_{\Sigma}$, that is, there is a bijection between the irreducible components of $\operatorname{Ext}(\phi)$, which are called   \emph{exceptional divisors} of $\phi$, and the junior elements of $G$ \cite{ItoReid94}. In particular, each irreducible component of $\operatorname{Ext}(\phi)$ is a smooth toric variety. Moreover, the complete
exceptional divisors of $\phi$ are those that are located in the variety $\phi^{-1}(V_{N}(\sigma))$, where $V(\sigma)=\overline{O_{N}(\sigma)}=O_{N}(\sigma)$. Therefore, $E_{g}$ is a complete exceptional divisor of $\phi$ if and only if $\hat{g}$ lies in the interior of $\sigma$.  
\end{remark}
\begin{example}
 The projective crepant resolution $$\phi_{(\hat{g_{1}},\hat{g_{2}},\hat{g_{3}},\hat{g_{4}})}:X_{\Sigma^{*}(\hat{g_{1}},\hat{g_{2}},\hat{g_{3}},\hat{g_{4}})}\longrightarrow\mathbb{C}^{3}/\mathbb{Z}_{6,(1,2,3)}$$ described in Example \ref{Z_{6}}  has a unique complete exceptional divisor $E_{g_{1}}$, since $\hat{g}_{1}$ is the unique element that lies in the interior of $\sigma$. Let us  compute $E_{g_{1}}$. Since $\{\hat{g_{1}},e_{2},e_{3}\}$ is a basis of $N$, we can identify $N$ with $\mathbb{Z}^{3}$ by changing the coordinates in the following way   
 \[
\hat{g_{1}} \;\mapsto\; (1,0,0),
\quad
e_{2} \;\mapsto\; (0,1,0),
\quad
e_{3} \;\mapsto\; (0,0,1).
\]
In these new coordinates, we have
\[
e_{1} = (6,-2,-3),
\quad
\hat{g_{2}} = (2,0,-1),
\quad
\hat{g_{3}} = (3,-1,-1),
\quad
\hat{g_{4}} = (4,-1,-2).
\]

Let $\rho_{g_{1}}$ be the cone generated by $\hat{g_{1}}$. In this case, we have 
\[
\begin{aligned}
N(\rho_{g_{1}}) 
&:=\mathbb{Z}^{3}/\mathbb{Z}\hat{g}_{1}\cong \mathbb{Z}^{2},
\\[6pt]
\operatorname{Star}(\rho_{g_{1}})
&= 
\bigl\{
\,\overline{\xi} \;\subseteq\; N(\rho_{g_{1}})_\mathbb{R}
\;\mid\;
\rho_{g_{1}} \preceq \xi \in \Sigma^{*}\bigl(\hat{g_{1}},\hat{g_{2}},\hat{g_{3}},\hat{g_{4}}\bigr)
\bigr\},
\\[6pt]
E_{g} 
&\;\cong\; 
X_{\operatorname{Star}(\rho_{g_{1}}),N(\rho_{g_{1}})}.
\end{aligned}
\] 
Since all the maximal cones of $\Sigma^{*}\bigl(\hat{g_{1}},\hat{g_{2}},\hat{g_{3}},\hat{g_{4}}\bigr)$ contain the ray $\rho_{g_{1}}$ and 
$$\overline{e_{1}}=(-2,-3), \quad \overline{\hat{g_{2}}}=(0,-1), \quad \overline{\hat{g_{3}}}=(-1,-1), \quad \overline{\hat{g_{4}}}=(-1,-2), \quad \overline{e_{2}}=(1,0), \quad \overline{e_{3}}=(0,1),$$
 it follows that the fan $\operatorname{Star}(\rho_{g_{1}})$ also has $6$ maximal cones, which are   
\[
\begin{alignedat}{2}
\xi_{1} &:= \operatorname{Cone}\bigl(\overline{e_{2}},\,\overline{e_{3}}\bigr), 
&\qquad
\xi_{2} &:= \operatorname{Cone}\bigl(\overline{e_{3}},\,\overline{\hat{g}_{3}}\bigr),\\[6pt]
\xi_{3} &:= \operatorname{Cone}\bigl(\overline{e_{1}},\,\overline{\hat{g}_{3}}\bigr), 
&\qquad
\xi_{4} &:= \operatorname{Cone}\bigl(\overline{e_{1}},\,\overline{\hat{g}_{4}}\bigr),\\[6pt]
\xi_{5} &:= \operatorname{Cone}\bigl(\overline{\hat{g}_{2}},\,\overline{\hat{g}_{4}}\bigr),
&\qquad
\xi_{6} &:= \operatorname{Cone}\bigl(\overline{e_{2}},\,\overline{\hat{g}_{2}}\bigr).
\end{alignedat}
\]
Thus, the representation of $\operatorname{Star}(\rho_{g_{1}})$ in $\mathbb{R}^{2}$ is given by the following picture

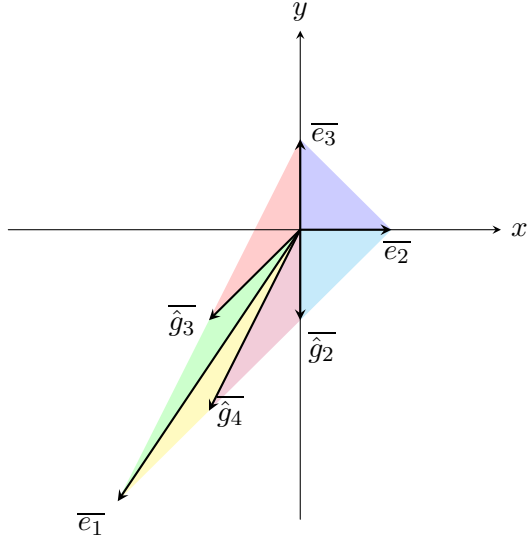
\begin{center}
\begin{tikzpicture}[scale=1.2, >=stealth, line cap=round, line join=round]
    \coordinate (O) at (0,0);

    \coordinate (e1) at (-2,-3);
    \coordinate (g2) at (0,-1);
    \coordinate (g3) at (-1,-1);
    \coordinate (g4) at (-1,-2);
    \coordinate (e2) at (1,0);
    \coordinate (e3) at (0,1);

    \fill[blue!20]   (O) -- (e2) -- (e3) -- cycle;     
    \fill[red!20]    (O) -- (e3) -- (g3) -- cycle;      
    \fill[green!20]  (O) -- (e1) -- (g3) -- cycle;      
    \fill[yellow!30] (O) -- (e1) -- (g4) -- cycle;      
    \fill[purple!20] (O) -- (g2) -- (g4) -- cycle;      
    \fill[cyan!20]   (O) -- (e2) -- (g2) -- cycle;      

    \draw[->,thick] (O) -- (e1)
        node[pos=1, anchor=north east, xshift=-1pt] {$\overline{e_{1}}$};
    \draw[->,thick] (O) -- (g2)
        node[pos=1, anchor=north west, xshift=-1pt] {$\overline{\hat{g}_{2}}$};
    \draw[->,thick] (O) -- (g3)
        node[pos=1, anchor=east, xshift=-1pt] {$\overline{\hat{g}_{3}}$};
    \draw[->,thick] (O) -- (g4)
        node[pos=1, anchor=west, xshift=-1pt] {$\overline{\hat{g}_{4}}$};
    \draw[->,thick] (O) -- (e2)
        node[pos=1, anchor=north, xshift=2pt] {$\overline{e_{2}}$};
    \draw[->,thick] (O) -- (e3)
        node[pos=1, anchor=west, yshift=2pt] {$\overline{e_{3}}$};

    \draw[->] (-3.2,0)--(2.2,0) node[right]{$x$};
    \draw[->] (0,-3.2)--(0,2.2) node[above]{$y$};

\end{tikzpicture}
\end{center}
This is a sequence of $3$ star subdivisions in the fan of $\mathbb{P}^{2}$. It can also be seen as a sequence of $3$ blowups starting in $\mathbb{P}^{2}$, with centers at $3$ points. Later, as consequence of our main result, we will see that $X_{\Sigma^{*}(\hat{g_{1}},\hat{g_{2}},\hat{g_{3}},\hat{g_{4}})}$ is total space of a  line bundle over $E_{g}$, and that explains why their fans have the same number of maximal cones. 
\end{example} All resolutions of $\mathbb{C}^{3}/\mathbb{Z}_{6,(1,2,3)}$
in  Example \ref{Z_{6}}  have a   graph with six triangles,  and this number is  the order of $\mathbb{Z}_{6,(1,2,3)}$. The next results show that this is no coincidence.
\begin{proposition} \cite[Thm.~12.3.9]{Cox11}
\label{p 2.8}
    Let $X_{\Xi}$ be an $n$-dimensional toric variety and let $$e(X_{\Xi}):= \sum_{i=0}^{2n}(-1)^{i}\rk H_{c}^{i}(X_{\Xi},\mathbb{Z})$$ be its topological Euler number. Then $e(X_{\Xi})=|\Xi(n)|$. 
\end{proposition}

\begin{proposition}\cite[Thm.~1.10] {Batyrev99}
Let $G$ be a finite subgroup of $\operatorname{SL}(n,\mathbb{C}).$ If a crepant resolution $X\rightarrow \mathbb{C}^{n}/G$ exists, then the Euler number of $X$
equals the number of conjugacy classes of $G$.
 \end{proposition}
It follows from the previous Propositions that when $G$ is an abelian finite subgroup of $\operatorname{SL}(n,\mathbb{C})$, and if there exists a toric crepant resolution $X_{\Sigma}\rightarrow \mathbb{C}^{n}/G$, then $|\Sigma(n)|=\#G$.

\subsection{Hilbert basis resolutions}

There are many examples of quotients $\mathbb{C}^{n}/G$ that do not admit a crepant resolution. The next result provides a big extent of such examples. 
\begin{proposition} \cite[Thm.~1.1]{Yamagishi18}
If $\mathbb{C}^{n}/G$ admits a (not necessarily projective) crepant
resolution, then $G$ is generated by junior elements.
\end{proposition}

Thus, for example, the quotient variety $\mathbb{C}^{4}/\mathbb{Z}_{2,(1,1,1,1)}$ does not admit crepant resolutions because $\mathbb{Z}_{2,(1,1,1,1)}$ contains only senior elements. The reciprocal of the above proposition is not true, as we shall see later on. Sometimes, for instance when $G$ is generated by junior elements,  it is possible to find another type of resolution, which we describe below, that has very important properties and provides another criterion for $\mathbb{C}^{n}/G$ to have a crepant resolution. 

In \cite[Thm.~16.4]{Schrijver86} it has been proved that, given a rational strongly convex polyhedral cone of maximal dimension $\xi$ in $L_{\mathbb{R}}$, then  $\xi\cap L$ has a unique minimal system of generators $\bold{Hlb}_{L}(\xi)$, which is given by
 $$\bold{Hlb}_{L}(\xi)=\{n\in \xi\cap L| \text{n is irreducible}\}$$
 where $n\in \xi\cap L$ is an irreducible element if $n=n_{1}+n_{2}$, with $n_{i}\in\xi\cap L$, then $n_{i}=n$ and $n_{j}=0$ for $i\neq j$.
 In particular, in  our case $\mathbb{C}^{n}/G=U_{\sigma,N}$, we have that $\nu_{G}\subset \mathbf{Hlb}_{N}(\sigma)$. Furthermore, since $N= \mathbb{Z}^{n}+\sum_{g\in G}\hat{g}\mathbb{Z}$, one has
 \begin{eqnarray}
  \mathbf{Hlb}_{N}(\sigma)\subset \hat{G}\cup\{e_{1},\dots,c_{n}\}.   
 \end{eqnarray}
 In fact, by Gordan Lemma  \cite[Prop.~1.2.17]{Cox11},
 $\sigma\cap N$ is generated by  union of the sets $\nu_{G}\setminus\hat{G}=\{e_{1},\dots,e_{n}\}$ and $K\cap N$, where 
 \[
K \;=\; 
\bigl\{
\,\sum_{m \in \nu_{G}\setminus\hat{G}} \delta_m \,m 
\;\bigm|\; 
0 \,\leq\, \delta_m \,<\, 1 \text{ for all } m\in \nu_{G}\setminus\hat{G}
\bigr\}.
\]
Note that $K=\operatorname{Conv}(0,e_{1},\dots,e_{n})$. Hence, by Proposition \ref{p.c.T}, $$K\cap N=\{\hat{g}\in \hat{G}|\operatorname{age}(g)=1\}.$$ Therefore, $\nu_{G}\setminus\hat{G}\cup (K\cap N)=\nu_{G}$ and the desired inclusion follows from the fact that $\mathbf{Hlb}_{N}(\sigma)$ is the minimal system of generators of $\sigma\cap N$.

\begin{definition}

A refinement $\Sigma$ of $\sigma$ defines a \emph{Hilbert basis resolution} $X_{\Sigma}\rightarrow U_{\sigma,N}$ if $\Sigma$ satisfies the following conditions:
\begin{itemize}
    \item $\Sigma$ is smooth.
    \item $\bold{Hlb}_{N}(\sigma)$ is the set of ray generators of $\Sigma$.
\end{itemize} 
\end{definition}

\begin{example}
    Consider the group $G:=\mathbb{Z}_{7, ((1, 1, 2, 3)}\subset \operatorname{SL}(4,\mathbb{C)}$.  In this case,
$\mathbb{C}^4 / G \;=\; U_{\sigma,\,N}$, where 
\[
   N \;=\; \mathbb{Z}^{3} \;+\; \mathbb{Z}\,\frac{1}{7}(1,1,2,3)
   \quad\text{and}\quad
   \sigma \;=\;\operatorname{Cone}\bigl(e_{1},\,e_{2},\,e_{3},\,e_{4}\bigr).
\]
Following the established notation, one has 
\begin{align}
\hat{G}= & \{(0,0,0,0), \quad \frac{1}{7}(1,1,2,3), \quad \frac{1}{7}(2, 2, 4, 6), \quad \frac{1}{7}(3, 3, 6, 2),  \\[3pt]  &  \frac{1}{7}(4, 4, 1, 5), \quad\frac{1}{7}(5, 5, 3, 1) , \quad\frac{1}{7}(6, 6, 5, 4)\}. \end{align}

Defining 
 \begin{equation}
\hat{g}_{1}  = \tfrac{1}{7}(1,1,2,3),\quad
\hat{g}_{2}  = \tfrac{1}{7}(3,3,6,2),\quad
\hat{g}_{3}  = \tfrac{1}{7}(4,4,1,5),\quad
\hat{g}_{4}  = \tfrac{1}{7}(5,5,3,1),
\end{equation}
 one can check that $$\bold{Hlb}_{N}(\sigma)=\{e_{1}, e_{1}, e_{3}, e_{4}, \hat{g_{1}}, \hat{g_{2}}, \hat{g_{3}}, \hat{g_{4}}\}.$$ The unique junior element $G$ corresponds to $\hat{g_{1}}$. Performing a star subdivision of $\sigma$ at $\hat{g_{1}}$, one gets a fan $\Sigma_{1}:=\Sigma^{*}(\hat{g_{1}})$, such that $$\Sigma_{1}(4)=\{\operatorname{Cone}(\hat{g_{1}},e_{2},e_{3},e_{4}), \operatorname{Cone}(\hat{g_{1}},e_{1}, e_{3},e_{4}), \operatorname{Cone}(\hat{g_{1}},e_{1},e_{2},e_{3}), \operatorname{Cone}(\hat{g_{1}}, e_{1},e_{2},e_{3})\}.$$ 
    This refinement provides a minimal model $X_{\Sigma_1}\rightarrow \mathbb{C}^{n}/G$, but this is not a resolution since $\operatorname{Cone}(\hat{g_{1}}, e_{1},e_{2},e_{3})$ is not a smooth cone. Furthermore, it is not possible to get a crepant resolution of $\mathbb{C}^{4}/G$ by a sequence of star subdivisions because $\nu_{G}=\{e_{1},e_{2},e_{3},e_{4},\hat{g_{1}}\}$. Nevertheless, it is possible to get a Hilbert basis resolution by performing the star subdivision of $\sigma$ at the sequence $(\hat{g_{1}},\hat{g_{2}},\hat{g_{3}},\hat{g_{4}})$. In this case, the resulting fan $\Sigma$ is such that
\[
\left\{
\begin{array}{l}
\operatorname{Cone}(\hat{g_{1}},e_{2},e_{3},e_{4}), \quad
\operatorname{Cone}(\hat{g_{1}},e_{1}, e_{3},e_{4}), \quad
\operatorname{Cone}(\hat{g_{1}},e_{1}, e_{2},\hat{g_{4}}), \quad
\operatorname{Cone}(\hat{g_{1}},e_{1}, e_{4},\hat{g_{4}}),
\\[6pt]
\operatorname{Cone}(\hat{g_{1}},e_{2}, e_{4},\hat{g_{4}}), \quad
\operatorname{Cone}(e_{1}, e_{2}, e_{4},\hat{g_{4}}), \quad
\operatorname{Cone}(\hat{g_{1}},e_{1}, e_{3},\hat{g_{3}}), \quad
\operatorname{Cone}(\hat{g_{1}},e_{2}, e_{3},\hat{g_{3}}),
\\[6pt]
\operatorname{Cone}(\hat{g_{1}},\hat{g_{3}}, e_{2},\hat{g_{5}}), \quad
\operatorname{Cone}(\hat{g_{1}},\hat{g_{3}}, e_{1},\hat{g_{5}}), \quad
\operatorname{Cone}(\hat{g_{1}},e_{1}, e_{2},\hat{g_{5}}),
\\[6pt]
\operatorname{Cone}(\hat{g_{5}},e_{1}, e_{2},e_{3}), \quad
\operatorname{Cone}(\hat{g_{3}},e_{2}, e_{3},\hat{g_{5}}), \quad
\operatorname{Cone}(\hat{g_{3}},e_{2}, e_{3},\hat{g_{5}})
\end{array}
\right.
\]
is the set of maximal cones of $\Sigma$, so that  $|\Sigma(4)|=14$.
\end{example}

The next result provides a necessary condition for the existence of crepant resolutions for an arbitrary Gorenstein quotient singularity $\mathbb{C}^{n}/G$ via the  Hilbert basis technique.
\begin{proposition}   \cite[Thm.~6.1]{Dais06} 
     Let $\mathbb{C}^{n}/G=U_{\sigma,N}$ be an abelian Gorenstein quotient singularity. If there exists a toric crepant resolution $X_{\Sigma}\rightarrow\mathbb{C}^{n}/G$ such that the set of ray generators of $\Sigma$ is $\nu_{G}$, then $\bold{Hlb}_{N}(\sigma)=\nu_{G}$.  
\end{proposition}
This  proposition says that whenever $\bold{Hlb}_{N}(\sigma)$ contains senior elements, the quotient $\mathbb{C}^{n}/G=U_{\sigma,N}$ does not admit a crepant resolution. In particular, $\mathbb{C}^{4}/\mathbb{Z}_{7, ((1, 1, 2, 3)}$ in fact does not have a crepant resolution, although $\mathbb{Z}_{7, ((1, 1, 2, 3)}$
is generated by a junior element. 
 
\bigskip
\section{Line bundles and exceptional divisors}

In this section we explore the structure of Hilbert basis resolutions and toric crepant resolutions of quotient singularities \(\mathbb{C}^n / G\) when  \( G \) is an abelian finite subgroup of \( \operatorname{SL}(n, \mathbb{C}) \). The results presented here extend the understanding of the exceptional sets of these resolutions and their relationship to junior elements of \( G \). We begin by proving a connection between the deformation retracts of the exceptional divisors corresponding to junior elements of $G$ and specific toric varieties associated with the fan of the resolution. More precisely, we show that if a junior element of \( G \) corresponds to a ray in the fan of a Hilbert basis resolution of $\mathbb{C}^{n}/G$, the associated irreducible component of the exceptional set is a deformation retract of an open toric subvariety (Theorem \ref{t3.5}). When the resolution is crepant, this result also establishes that each exceptional divisor is normally embedded and the total space of its normal bundle is isomorphic to an open toric subvariety of the resolution.

Furthermore, by combining these results, we describe the global structure of the resolution. For \( n = 3 \), we identify compact junior elements of \( G \), and show that a toric crepant resolution of \(\mathbb{C}^3 / G\) is a toric glueing of the total spaces of the canonical bundles over the exceptional divisors associated with these compact junior elements (Corollary \ref{cor3.7}). A special case arises when there is only one compact junior element, where the resolution itself is isomorphic to the total space of the canonical bundle over the corresponding exceptional divisor (Corollary  \ref{cor3.8}).

\subsection{Neighborhood retracts and  tubular neighborhoods}
We start by recalling the definition of tubular neighborhoods in the category of complex manifolds.  
\begin{definition} Let $E$ be a submanifold 
of a complex manifold $X$.
 \begin{itemize}
     \item[(a)]  \( E \) is a \emph{holomorphic neighborhood retract} of \( X \) if there is a neighborhood \( U \) of \( E \) and a holomorphic map \( h : U \to E \) such that \( h|_E = \text{identity} \);
     
\item[(b)] \( E \) is \emph{normally embedded} in \( X \) if there is a neighborhood \( U_0 \) of the zero section \( Z_{E,X} \) of the normal bundle of $E$ with respect to $X$, \( N_{E,X} \), and a biholomorphic map of \( U_0 \) onto a neighborhood $U$ of \( E \) in \( X \) which is the identity on \( Z_{E,X} \). In this case $U$ is called of \emph{tubular neighborhood} of $E$.
 \end{itemize}

\end{definition}
For more details about the above definitions see \cite{MorrowRossi78}. In the category of real smooth manifolds, tubular neighborhoods always exist,
 while this is not the case in the category of complex manifolds, as  the normal bundle sequence
$$0 \to T_E \to T_X | _E \to N_{E, X} \to 0$$
may not split (see \cite{AbateBracciTovena09}); here $T_{X}$ is the tangent bundle of $X$ and $N_{E,X}$ is the normal bundle of $E$ with respect to $X$.

\subsection{Exceptional divisors and their algebraic tubular neighborhoods}
We will prove the existence of  ``algebraic neighborhood retracts" in the following context. Let $G$ be a finite abelian subgroup of $\operatorname{SL}(n,\mathbb{C})$ and let $\{g_{1},\dots,g_{s}\}$ be the set of junior classes of $G$. Suppose that there is a Hilbert basis resolution $f:X_{\Sigma}\to\mathbb{C}^{n}/G=U_{\sigma,N}$ of the quotient variety $\mathbb{C}^{n}/G$. By Proposition \ref{p.c.T}, $\hat{g_{1}},\dots,\hat{g}_{s}$ are elements of $\bold{Hlb}_{N}(\sigma)$. Thus, $E_{g_{i}};=V(\operatorname{Cone}(\hat{g_{i}}))$ is a exceptional prime divisor of the resolution for every $i=1,\dots, s$. We will see that there is an open toric set $U_{i}$ of $X_{\Sigma}$ containing $E_{g_{i}}$ together with a torus invariant divisor $D_{i}$ of $E_{g_{i}}$ such that one finds an isomorphism $\varphi: \operatorname{tot}(\mathcal{O}_{E_{g_{i}}}(D_{i}))\to U_{i}$ that is the identity in the zero section. In particular, when $f$ is a crepant resolution of $\mathbb{C}^{n}/G$, the variety $X_{\Sigma}= \bigcup_{i=1}^{s}U_{i}$,  $\mathcal{O}_{E_{g_{i}}}(D_{i})$ is   the normal bundle $\mathcal{N}_{E_{g_{i}}/X_{\Sigma}}$ of $E_{g_{i}}$ in $X_{\Sigma}$, and $U_{i}$ is an algebraic  tubular neighborhood of $E_{g_{i}}$.

As a warm-up we   start with an example that somewhat recaps what has been said so far.
\begin{example}\label{Ex3.1}
Let $G$ be the cyclic group $\mathbb{Z}_{5,(1,2,2)}\subset \operatorname{SL}(3,\mathbb{C})$. In this case $$\hat{G}=\{(0,0,0), \frac{1}{5}(1,2,2), \frac{1}{5}(2,4,4), \frac{1}{5}(3,1,1),\frac{1}{5}(4,3,3)\}$$ 
and $\mathbb{C}^{3}/G=U_{\sigma,N}$, where $N=\mathbb{Z}^{3}+\mathbb{Z}\frac{1}{5}(1,2,2)$. Note that $\{\frac{1}{5}(1,2,2),\frac{1}{5}(3,1,1),(0,0,1)\}$ is a basis of $N$ and 
\begin{align}
 (1,0,0)&=-(\frac{1}{5}(1,2,2))+2(\frac{1}{5}(3,1,1))+0(0,0,1)\\
 (0,1,0)&=\ 3(\frac{1}{5}(1,2,2))+(-1)(\frac{1}{5}(3,1,1))+(-1)(0,0,1).
\end{align}

Thus, after a change of    coordinates, one can assume $\mathbb{C}^{4}/G=U_{\xi,\mathbb{Z}^{3}}$, where $$\xi=\operatorname{Cone}(u_{1}:=(-1,2,0), u_{2}:=(3,-1,-1),e_{3}).$$ Moreover, the junior elements of $G$ are given as $\hat{g_{1}}=e_{1}$ and $\hat{g_{2}}=e_{2}$. This way, after performing a star subdivision of $\xi$ at the sequence $(\hat{g}_{1},\hat{g_{2}})$ one has the following picture 
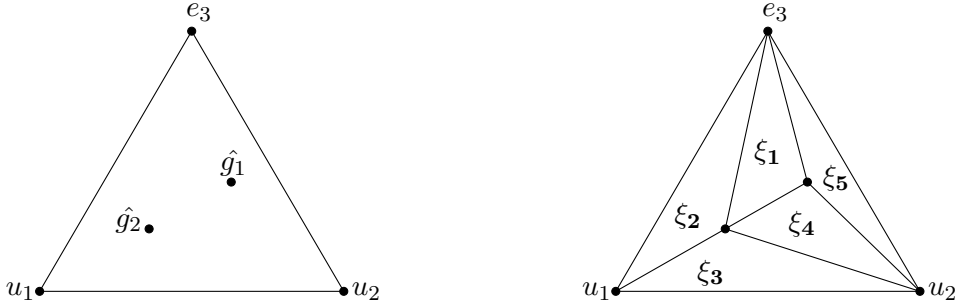
\begin{center}
\begin{tikzpicture}[scale=0.50]
 	\path  (0,0) to (4,2.3) to  (4,6.9) to (0,0) ;
 	\path  (8,0) to (4,2.3) to  (4,6.9) to (8,0) ;
 	\path  (0,0) to (4,2.3)  to  (4,0) to (0,0);
 	\path (4,0) to (4,2.3)  to  (8,0) to (4,0);
 	\path (1.5,0) to (4,2.3)  to  (6.5,0) to (1.5,0);
 	\draw [fill] (0,0) circle (3pt);
 	\draw [fill]  (8,0) circle (3pt);
 	\draw [fill] (4,6.9) circle (3pt);
 	\draw [fill] (5.04,2.898) circle (3pt);
 	
    \draw [fill] (2.88,1.656) circle (3pt);
 	
 	\draw (0,0) -- (8,0); \draw (0,0) -- (4,6.9); \draw (8,0) --
 	(4,6.9); \node at (-0.5,0) {$u_1$}; \node at (8.6,0)
 	{$u_2$}; \node  at (4.2,7.4)
 	{$ e_3$}; \node at
 	(5.1,3.4) {$\hat{g_{1}}$};\node  at (2.35,1.9)
 	{$ \hat{g_{2}}$};
 	 
\end{tikzpicture}
\hskip1cm
 \raisebox{5em}{
 \begin{tikzpicture}
     $\xrightarrow{(\hat{g_{1}},\hat{g_{2}})}$
 \end{tikzpicture}} \hskip1cm
 \begin{tikzpicture}[scale=0.50]
 	\path  (0,0) to (4,2.3) to  (4,6.9) to (0,0) ;
 	\path  (8,0) to (4,2.3) to  (4,6.9) to (8,0) ;
 	\path  (0,0) to (4,2.3)  to  (4,0) to (0,0);
 	\path (4,0) to (4,2.3)  to  (8,0) to (4,0);
 	\path (1.5,0) to (4,2.3)  to  (6.5,0) to (1.5,0);
 	\draw [fill] (0,0) circle (3pt);
 	\draw [fill]  (8,0) circle (3pt);
 	\draw [fill] (4,6.9) circle (3pt);
    \draw [fill] (5.04,2.898) circle (3pt);
 	
    \draw [fill] (2.88,1.656) circle (3pt);
 	
 	\draw (0,0) -- (8,0); \draw (0,0) -- (4,6.9); \draw (8,0) --
 	(4,6.9);\draw (0,0) -- (5.04,2.898); \draw (8,0) -- (5.04,2.898); \draw (4,6.9) -- (5.04,2.898);\draw (4,6.9) -- (2.88,1.656); \draw (8,0) -- (2.88,1.656); \node at (-0.5,0) {$u_1$}; \node at (8.6,0)
 	{$u_2$}; \node  at (4.2,7.4)
 	{$ e_3$}; \node at
 	(5.8,3) {$\bold{\xi_{5}}$};\node  at (4,3.7)
 	{$\bold{\xi_{1}}$}; \node  at (1.9,2)
 	{$\bold{\xi_2}$};\node  at (2.5,0.5)
 	{$\bold{\xi_3}$};\node  at (4.9,1.8)
 	{$\bold{\xi_4}$};
	
 \end{tikzpicture} 
\end{center}
Let $\phi:X_{\Xi}\longrightarrow \mathbb{C}^{3}/G$ the corresponding crepant resolution and let $\rho_{1}=\operatorname
{Cone}(\hat{g_{1}}), \rho_{2}=\operatorname
{Cone}(\hat{g_{2}}), \rho_{3}=\operatorname
{Cone}(e_{3}), \rho_{4}=\operatorname
{Cone}(u_{1}), \rho_{5}=\operatorname
{Cone}(u_{2})$ be the rays of $\Xi$. Note that $Exc(\phi)=E_{g_{1}}\cup E_{g_{2}}$, where $E_{g_{i}}:=D_{\rho_{i}}=\overline{O(\rho_{i})}$ for $i=1,2$. Furthermore, one can check that $E_{g_{1}}\cong\mathbb{P}^{2}$ and $E_{g_{2}}\cong \mathbb F_{3}$,   the third Hirzebruch surface. Now let us describe $X_{\Xi}$ in homogeneous coordinates. Since $X_{\Xi}$ is smooth and its fan has a cone of maximal dimension, it follows that $\operatorname{Cl}(X_{\Xi})=\operatorname{Pic}(X_{\Xi})$ is a free group; morever, since $X_{\Xi}$ has no torus factors, the   sequence    $$\xymatrix{0\ar[r]&\mathbb{Z}^{3}\ar[r]^{\varphi}&\bigoplus_{i=1}^{5}D_{\rho_{i}}\ar[r]^{\psi}&\operatorname{Pic}(X_{\Xi})\ar[r]&0} $$
is exact. Moreover, 
 \begin{eqnarray*}
      \operatorname{div} (\chi^{(1,0,0)})= & D_{\rho_{1}}-D_{\rho_{4}}+3D_{\rho_{5}} \\
      \operatorname{div} (\chi^{(0,1,0)})= & D_{\rho_{2}}+2D_{\rho_{4}}-D_{\rho_{5}} \\
      \operatorname{div} (\chi^{(0,0,1)})= & D_{\rho_{3}}-D_{\rho_{5}}.
 \end{eqnarray*}
 Thus $[D_{\rho_{1}}]=[D_{\rho_{4}}]-3[D_{\rho_{5}}], [D_{\rho_{2}}]=-2[D_{\rho_{4}}]+[D_{\rho_{5}}], [D_{\rho_{3}}]=[D_{\rho_{5}}]$ in $\operatorname{Pic}(X_{\Xi})$. In particular, $\{[D_{\rho_{4}}],[D_{\rho_{5}}]\}$ is basis of $\operatorname{Pic}(X_{\Xi})$, and   one can identify $\operatorname{Pic}(X_{\Xi})\cong\mathbb{Z}^{2}$. This way, the previous exact sequence can be rewritten as 
 $$\xymatrix{0\ar[r]&\mathbb{Z}^{3}\ar[r]^{\varphi}&\mathbb{Z}^{5}\ar[r]^{\psi}&\mathbb{Z}^{2}\ar[r]&0} $$
 where $\varphi(m_{1},m_{2},m_{3})=(m_{1},m_{2},m_{3},-m_{1}+2m_{2},3m_{1}-m_{2}-m_{3})$ and $\psi(n_{1},n_{2},n_{3}n_{4},n_{5})=(n_{1}-2n_{2}+n_{4},-3n_{1}+n_{2}+n_{3}+n_{5})$. Applying the functor $\Hom(-,\mathbb{C}^{*})$ to the latter sequence, one gets
 $$\xymatrix{0\ar[r]&(\mathbb{C}^{*})^{2}\ar[r]^{\psi^{*}}&\ (\mathbb{C}^{*})^{5}\ar[r]^{\varphi^{*}}&(\mathbb{C}^{*})^{3}\ar[r]&0} $$
 where $\psi^{*}(t_{1},t_{2})=(t_{1}t_{2}^{-3},t_{1}^{-2}t_{2},t_{2},t_{1},t_{2})$ and $\varphi^{*}(q_{1},q_{2},q_{3},q_{4},q_{5})=(q_{1}q_{4}^{-1}q_{5}^{3},q_{2}q_{4}^{2}q_{5}^{-1},q_{3}q_{5}^{-1})$. Notice that $\psi^{*}$ defines a action of $(\mathbb{C}^{*})^{2}$ on $\mathbb{C}^{5}$. Let $S=\mathbb{C}[x_{1},x_{2},x_{3},x_{4},x_{5}]$ be total coordinate ring of $X_{\Xi}$, where each variable $x_{i}$ corresponds to the ray $\rho_{i}$ of $\Xi$. $S$ is a graded ring such that $\deg(x_{i})=[D_{\rho_{i}}]\in \operatorname{Pic}(X_{\Xi})$. Consider the  monomial ideal $$B(\Xi):=(x_{4}x_{5},x_{1}x_{5},x_{1}x_{3}, x_{3}x_{4},x_{2}x_{4})$$
 and its zero locus $Z(\Xi):=V(B(\Xi))\subset\mathbb{C}^{5}$. Given polynomial $g\in S$, denote by $D(g)$ the principal open set of $\mathbb{C}^{5}$ where $g\neq 0$. Note that 
 $$\mathbb{C}^{5}\backslash Z(\Xi)=D(x_{4}x_{5})\cup D(x_{1}x_{5})\cup D(x_{1}x_{3})\cup D(x_{3}x_{4})\cup D(x_{2}x_{4})$$
 and each principal open set  in this union is invariant under the action of $(\mathbb{C}^{*})^{2}$. One has  $$(\mathbb{C}^{5}\backslash Z(\Xi)/(\mathbb{C}^{*})^{2})\cong X_{\Xi}$$ and each homogeneous ideal of $S$ which is contained in $B(\Xi)$ provides a closed subset of  $X_{\Xi}$. One has 
 \begin{align}
     X_{\Xi}=&\{[x_{1}:x_{2}:x_{3}:x_{4}:x_{5}]|(x_{1},x_{2},x_{3},x_{4},x_{5})\in\mathbb{C}^{5}\backslash Z(\Xi)\} \\
     E_{g_{1}}=&\{[0:x_{2}:x_{3}:x_{4}:x_{5}]|(0,x_{2},x_{3},x_{4},x_{5})\in D(x_{4}x_{5})\cup D(x_{3}x_{4})\cup D(x_{2}x_{4})\}\\
     E_{g_{2}}=&\{[x_{1}:0:x_{3}:x_{4}:x_{5}]|(x_{1},0,x_{3},x_{4},x_{5})
     \\ &  \ \  \ \ \in D(x_{4}x_{5})\cup D(x_{1}x_{5})\cup D(x_{1}x_{3})\cup D(x_{3}x_{4})\} \\
     X_{\Xi_{g_{1}}}:=&\{[x_{1}:x_{2}:x_{3}:x_{4}:x_{5}]|(x_{1},x_{2},x_{3},x_{4},x_{5})
     \\ & \ \  \ \   \in D(x_{4}x_{5})\cup D(x_{3}x_{4})\cup D(x_{2}x_{4})\}\\
     X_{\Xi_{g_{2}}}:=&\{[x_{1}:x_{2}:x_{3}:x_{4}:x_{5}]|(x_{1},x_{2},x_{3},x_{4},x_{5})\\ &  \ \  \ \  \in D(x_{4}x_{5})\cup D(x_{1}x_{5})\cup D(x_{1}x_{3})\cup D(x_{3}x_{4})\}.
 \end{align}
 
 Note that each $X_{\Xi_{g_{i}}}$ is an open set   containing $E_{g_{i}}$. Furthermore, the projections \[
\pi_{i}:X_{\Xi_{g_{i}}} \longrightarrow E_{\hat{g_{1}}}
\]
defined by
\[
\pi_{1}([x_{1}:x_{2}:x_{3}:x_{4}:x_{5}]) = [0:x_{2}:x_{3}:x_{4}:x_{5}]
\]
\[
\pi_{2}([x_{1}:x_{2}:x_{3}:x_{4}:x_{5}]) = [x_{1}:0:x_{3}:x_{4}:x_{5}]
\] are well defined morphisms which are the identity in $E_{g_{i}}$. Therefore, $E_{g_{i}}$ is a algebraic neighborhood retract of $X_{\Xi}$.     
\end{example}

 Given a Cartier divisor $D=\sum_{\rho\in\Xi(1)}a_{\rho}D_{\rho}$ of a toric variety $X_{L,{\Xi}}=X_{\Xi}$, the following construction will be important for our goals. Let $u_{\rho}$ be the minimal rational generator of $\rho\in \Xi(1)$. For each $\tau\in\Xi$ consider the cone in $L_{\mathbb{R}}\times\mathbb{R}$ defined by
   $$\tilde{\tau}:=\operatorname{Cone}((0,1),(u_{\rho},-a_{\rho})|u_{\rho}\in\tau).$$
   Note that $\tilde{\tau}$ is a strongly convex rational polyhedral cone in $L_{\mathbb{R}}\times\mathbb{R}$. Let $\Xi_{D}$ be the set consisting of the cones $\tilde{\tau}$ and their faces, for $\tau\in\Xi$. This is a fan  in $L_{\mathbb{R}}\times\mathbb{R}$ and the projection $L_{\mathbb{R}}\times\mathbb{R}\longrightarrow L$ is a compatible morphism for the fans $\Xi_{D}$ and $\Xi$. Hence, we get a toric morphism $$\pi:X_{\Xi_{D}}\longrightarrow X_{\Xi}.$$ From this construction one has the following result  \cite[Prop.~7.3.1]{Cox11}.
   \begin{proposition}
       
  \label{prop 3}
   	$\pi:X_{\Xi_D}\longrightarrow X_{\Xi}$ is a line  bundle whose sheaf of sections is $\mathcal{O}_{X_{\Xi}}(D)$.
   \end{proposition}
  In particular, this proposition implies that the total space of the line bundle corresponding to a  Cartier divisor of a toric variety is also a toric variety.
  
  \begin{lemma}\label{lema 1}
 Let $\{u_{1},\dots,u_{n}\}$ be a basis of  $\mathbb{R}^{n}$. If $b_{i}\in \operatorname{Conv}(k_{i}u_{1},\dots,k_{i}u_{n})$, for $i=1,\dots, n$ and $k_{i}\in \mathbb{Z}_{>0}$, are such that $\{b_{1},\dots,b_{n}\}$ is another basis of $\mathbb{R}^{n}$ and  $b\in \operatorname{Conv}(ku_{1},\dots,ku_{n})$, for some $k\in \mathbb{Z}_{>0}$, and $b=a_{1}b_{1}+\dots+a_{n}b_{n}$, then $k_{1}a_{1}+\dots+k_{n}a_{n}=k$.  
  \end{lemma}
\begin{proof}
	One can write $b=\sum_{i=1}^{n}\alpha_{i}ku_{i}$, where $\sum_{i=1}^{n}\alpha_{i}=1$, and $b_{j}=\sum_{i=1}^{n}a_{ij}k_{j}u_{i}$, where $\sum_{i=1}^{n}a_{ij}=1$ . The latter implies that $b=\sum_{i=1}^{n}(\sum_{j=1}^{n}a_{j}k_{j}a_{ij})u_{i}$. Thus $\alpha_{i}k=\sum_{j=1}^{n}k_{j}a_{j}a_{ij}$. It follows that 
	$$k=\sum_{i=1}^{n}\alpha_{i}k=(\sum_{i=1}^{n}a_{i1})k_{1}a_{1}+\dots+(\sum_{i=1}^{n}a_{in})k_{n}a_{n}=k_{1}a_{1}+\dots+k_{n}a_{n}$$
	
\end{proof}

   Now we can prove our first   result.

\begin{theorem}
    \label{t3.5}
Let $G$ be an abelian finite subgroup of $\operatorname{SL}(n,\mathbb{C})$. Suppose that there is a Hilbert desingularization $\phi:X_{\Xi}\rightarrow \mathbb{C}^{n}/G=U_{\sigma,N}$ of the quotient variety, and for each $\hat{g}\in \hat{G}\cap \bold{Hlb}_{N}(\sigma)$ denote by $E_{g}$  the irreducible component of the exceptional set of $\phi$, corresponding to the ray $\rho_{g}:=\operatorname{Cone}(\hat{g})$. If $\operatorname{age}(\hat{g})=1$ then $E_{g}$ is a deformation retract of the toric variety $X_{\Xi_{g}}$ , where $\Xi_{g}$ is the fan that consists of all the faces of the cones that appear in the set $$\Xi_{g}(n):=\{\eta\in\Xi(n)|\rho_{g}\preceq\eta\}.$$ In particular $X_{\Xi_{g}}$ is open in $X_{\Xi}$. 
\end{theorem}
\begin{proof}
    Consider the   lattice $N(\hat{g})=N/\hat{g}\mathbb{Z}\cong\mathbb{Z}^{n-1}$ and the set $$\operatorname{Star}(g):=\{\overline{\eta}\subseteq N(\hat{g})| \rho_{g}\preceq\eta\in\Xi\}.$$
Note that $\operatorname{Star}(g)$ is a fan in $N(\hat{g})_{\mathbb{R}}$ and it is actually the fan of $E_{g}$. Furthermore, all elements of $\operatorname{Star}(g)(1)$ are of the form $\overline{\operatorname{Cone}(\hat{g},u)}=\operatorname{Cone}(\overline{u})$, for some $u\in \bold{Hlb}_{N}(\sigma)$ and $\operatorname{Cone}(\hat{g},u)\in\Xi_{g}(2)$.  Since $E_{g}$ is a smooth variety any of its Weil divisors are Cartier. Consider the Cartier divisor $$D:=\sum_{\operatorname{Cone}(\overline{u})\in \operatorname{Star}(g)(1)}-(\operatorname{age}(u))D_{\operatorname{Cone}(\overline{u})},$$ where $D_{\operatorname{Cone}(\overline{u})}$ is the toric prime divisor corresponding to the ray $\operatorname{Cone}(\overline{u})$. Let $\xi:=\operatorname{Cone}(\hat{g},u_{1},\dots,u_{n-1})$ be a cone in $\Xi(n)$. This way, $\{\overline{u_{1}}\dots,\overline{u_{n-1}}\}$ is a basis of $N(g)$ and hence $\{(\overline{0},1), (\overline{u_{1}}, \operatorname{age}(u_{1})),\dots, (\overline{u_{n-1}}, \operatorname{age}(u_{n-1}))\}$ is a basis of $N(g)\times \mathbb{Z}$. Consider the lattice isomorphism 
\[
\overline{f} : N(g) \times \mathbb{Z} \longrightarrow N
\]
defined by
\[
\overline{f}(\overline{0}, 1) = \hat{g}, \quad 
\overline{f}(\overline{u_1}, \mathrm{age}(u_1)) = u_1, \dots, \quad 
\overline{f}(\overline{u_{n-1}}, \mathrm{age}(u_{n-1})) = u_{n-1}.
\]
Let $\rho$ be an element of $\Xi_{g}(1)$ and denote by $u_{\rho}$ its ray generator and $k_{\rho}$ the age of $u_{\rho}$. Note that $u_{\rho}\in \operatorname{Conv}(k_{\rho}e_{1},\dots,k_{\rho}e_{n})$. One can write $u_{\rho}=a\hat{g}+a_{1}u_{1}+\dots+a_{n-1}u_{n-1}$, hence, by Lemma \ref{lema 1}, $$k_{\rho}=a(\operatorname{age}(\hat{g}))+a_{1}(\operatorname{age}(u_{1}))+\dots+a_{n-1}(\operatorname{age}(u_{n-1})).$$ Thus, one gets 
\begin{align}
	\overline{f}(\overline{u_{\rho}},k_{\rho}) & = \overline{f}(\overline{a\hat{g}+a_{1}u_{1}+\dots+a_{n-1}u_{n}},a(\operatorname{age})+a_{1}(\operatorname{age}(u_{1}))\dots+a_{n-1}(\operatorname{age}(u_{n-1})) \\ &
	= \overline{f}(\overline{0},1)+a_{1}f(\overline{u_{1}},\operatorname{age}(u_{1}))+\dots +a_{n-1}\overline{f}(\overline{u_{n-1}},\operatorname{age}(u_{n-1})) \\ & 
	= u_{\rho}.&
\end{align}
Let \( \eta = \mathrm{Cone}(v_1, \dots, v_s) \) be a cone in \( \Xi_g \), where \( v_i \in \bold{Hlb}_{N}(\sigma) \). If \( \hat{g} \in \eta \), then we can suppose \( u_1 = \hat{g} \), and we have 
\[
\overline{n} \in \operatorname{Star}(g), \quad \hat{\overline{\eta}} \in \operatorname{Star}(g)_D,
\]
so that
\[
\overline{f}_{\mathbb{R}}(\hat{\overline{\eta}}) = \eta.
\]

If \( \hat{g} \not\in \eta \), then \( \xi := \rho_g + \eta \) is a cone in \( \Xi_g \) such that \( \eta \) is one of its faces. In this case,
\[
\hat{\overline{\xi}} = \mathrm{Cone}((\overline{0}, 1), (u_1, \mathrm{age}(u_1)), \dots, (u_n, \mathrm{age}(u_n)))
\]
is a cone in \( \operatorname{Star}(g)_D \). Thus, its face 
\[
\eta' := \mathrm{Cone}((u_1, \mathrm{age}(u_1)), \dots, (u_n, \mathrm{age}(u_n)))
\]
is also in \( \operatorname{Star}(g)_D \). Moreover, 
\[
\overline{f}_{\mathbb{R}}(\eta') = \eta.
\]

The two cases analyzed above show that \( \overline{f} \) is compatible with the fans \( \operatorname{Star}(g)_D \) and \( \Xi_g \), and hence, one obtains an isomorphism of toric varieties
\[
f: X_{\operatorname{Star}(\hat{g})_{K_{E_g}}} \longrightarrow X_{\Xi_g}.
\]

The last statement implies that \( X_{\Xi_g} \) is a line bundle over its subset \( E_g \), and therefore, \( E_g \) is a (strong) deformation retract of \( X_{\Xi_g} \).
\end{proof}

  \begin{theorem}\label{t3.4} Let $G$ be an abelian finite subgroup of $\operatorname{SL}(n,\mathbb{C})$ and suppose that there is a crepant resolution $\mathbb{\phi}:X_{\Xi}\rightarrow \mathbb{C}^{n}/G=U_{\sigma,N} $ of the quotient variety. If $g$ is a junior element of $G$, then   $E_{g}$ is normally embedded in $X_{\Xi}$. In particular, the total space of the canonical bundle of $E_{g}$, $\operatorname{tot}(\omega_{E_{g}})$, is isomorphic to the toric variety $X_{\Xi_{g}}$, and 
    $$X_{\Xi}=\bigcup_{\hat{g}\in\hat{G}\cap\bigtriangleup_{1}}X_{\Xi_{g}}$$
    where $X_{\Xi_{g}}$ is defined in the same way as the previous theorem. 
  \end{theorem}
\begin{proof}
The proof is very similar to that of the previous theorem up to some adjustments. Again, one considers the lattice $N(g)$ and the fan $\operatorname{Star}(g)$ of $E_{g}$. Since $E_{g}$ is a smooth variety, its canonical divisor  $K_{E_{g}}$ is Cartier.  Since $\phi$ is a crepant morphism, the age of all the ray generators of the rays in $\Xi$ should equal $1$. This way, one has 
$$\sum_{\operatorname{Cone}(\overline{u})\in \operatorname{Star}(g)(1)}-(\operatorname{age}(u))D_{\operatorname{Cone}(\overline{u})}=\sum_{\operatorname{Cone}(\overline{u})\in \operatorname{Star}(g)(1)}-D_{\operatorname{Cone}(\overline{u})}=K_{E_{g}}.$$
Now consider the fan $\operatorname{Star}(g)_{K_{E_{g}}}$ in  $(N(g)\times \mathbb{Z})_{\mathbb{R}}$. By Proposition \ref{prop 3}, $X_{\operatorname{Star}(g)_{K_{E_{g}}}}$ is the total space of  $\mathcal{O}_{E_{g}}(K_{E_{g}})=\omega_{E_{g}}$. Let $\xi:=\operatorname{Cone}(\hat{g},u_{1},\dots,u_{n-1})$ be a cone in $\Xi(n)$. This way, $\{\overline{u_{1}}\dots,\overline{u_{n-1}}\}$ is a basis of $N(g)$ and hence $\{(\overline{0},1), (\overline{u_{1}}, \operatorname{age}(u_{1})),\dots, (\overline{u_{n-1}}, \operatorname{age}(u_{n-1}))\}$ is a basis of $N(g)\times \mathbb{Z}$. Consider the  isomorphism $\overline{f}:N(g)\times \mathbb{Z}\longrightarrow N$  of lattices defined from $\overline{f}(\overline{0},1)=g$, $\overline{f}(\overline{u_1},1)=u_{1}$,\dots, $\overline{f}(\overline{u_{n-1}},1)=u_{n-1}$. Write $u_{\rho}=a_{1}g+a_{2}v_{1}+\dots+a_{n}v_{n-1}$. Since $u_{\rho},g,v_{1},\dots v_{n-1}\in T_{\sigma}$, Lemma \ref{lema 1} guarantees $a_{1}+\dots+a_{n}=1$. With the same arguments of the proof of the last theorem, one gets that $\overline{f}_{\mathbb{R}}$ induces an isomorphism of toric varieties $f:X_{\operatorname{Star}(\hat{g})_{K_{E_{g}}}}\longrightarrow X_{\Xi_{g}}$.

Since $\phi$ is a crepant resolution of $\mathbb{C}^{n}/G$ then $\omega_{X_{\Xi}}\cong \mathcal{O}_{X_{\Xi}}$. By the adjunction formula 
\[
\omega_{E_{g}} \simeq \omega_{X_{\Xi}} \otimes_{\mathcal{O}_{X_{\Xi}}} \mathcal{N}_{E_{g}/X_{\Xi}},
\]
one has that $\omega_{E_{g}}\cong \mathcal{N}_{E_{g}/X_{\Xi}}$. Therefore $E_{g}$ is normally embedded in $X_{\Xi}$.
\end{proof}

\begin{corollary} \label{cor3.7}
	Let $G$ be an abelian finite subgroup of $\operatorname{SL}(3,\mathbb{C})$. Suppose $G$ has compact junior elements $g_{c_{1}},\dots, g_{c_{s}}$ (those elements whose fractional expressions lie in $\operatorname{Relint}(\sigma))$. Then there is a toric minimal model $X_{\Xi}$ of $\mathbb{C}^{3}/G=U_{\sigma,N}$  such that $X_{\Xi}$ is a toric glueing of $\operatorname{tot}(\omega_{E_{g_{c_{1}}}}),\dots, \operatorname{tot}(\omega_{E_{g_{c_{s}}}})$.  
\end{corollary}

\begin{proof}
 Since in dimension $3$ every Gorenstein toric variety with terminal singularities is smooth, every sequence of star subdivisions at all the points of $\nu_{G}$ provides a crepant resolution of $\mathbb{C}^{3}/G=U_{\sigma,N}$. This way, if one starts a sequence of star subdivisions first at the points $\hat{g_{c_{1}}},\dots, \hat{g_{c_{s}}}$, and then at    the   remaining points of $\nu_G$, one gets a fan $\Xi$ such that any of its maximal cones  contains the ray $\operatorname{Cone}(\hat g_{c_{i}})$ for some $i=1,\dots,s$. From the previous Corollary, it follows that $X_{\Xi}=\bigcup_{i=1}^{s}X_{\Xi_{g_{c_{i}}}}$.
\end{proof}

\begin{corollary} \label{cor3.8}
		Let $G$ be an abelian finite subgroup of $\operatorname{SL}(3,\mathbb{C})$. Then $G$   has only one compact junior element $g$ if and only if $\mathbb{C}^{3}/G=U_{N,\sigma}$ has a toric crepant resolution $X_{\Xi}$ that is isomorphic to $\operatorname{tot}(\omega_{E_{g}})$.
\end{corollary}
\begin{proof}
One direction of the proof follows from the previous corollary. For the opposite implication, suppose that such a resolution $X_{\Xi}$ exists. Since  $X_{\Xi}$ is a line bundle of $E_{g}$, by Poincaré duality one gets  $H_{c}^{4}(X_{\Xi},\mathbb{Q})\cong H_{2}(X_{\Xi},\mathbb{Q})\cong H_{2}(E_{g},\mathbb{Q})\cong H^{0}(E_{g},\mathbb{Q})$. As $E_{\hat{g}}$ is a complete toric surface,   $\dim H^{0}(E_{g},\mathbb{Q})=1$. Thus $\dim H_{c}^{4}(X_{\Xi},\mathbb{Q})=1$. By \cite{ItoReid94}, $\dim H_{c}^{4}(X_{\Xi} ,\mathbb{Q})$ is the number of compact junior elements of $G$.   
\end{proof}
 


\bigskip\frenchspacing
\def\cprime{$'$}

\end{document}